\documentclass[12pt,a4paper]{article}
 
\usepackage[latin1]{inputenc}
\usepackage[english]{babel}
\usepackage[T1]{fontenc}
\usepackage{amsmath} 
\usepackage{amsfonts}
\usepackage{amssymb}
\usepackage{amsthm}
\usepackage{dsfont}
\usepackage{stmaryrd}
\usepackage{tikz}
\usepackage{todonotes}
\usetikzlibrary{matrix,arrows.meta}
\usepackage[title,titletoc,toc]{appendix}

\allowdisplaybreaks

\usepackage{hyperref}
\usepackage{xcolor}
\usepackage[a4paper,top=2.1cm,bottom=2.60cm,left=2.1cm,right=2.1cm]{geometry} 

\theoremstyle{plain}
\newtheorem{theorem}{Theorem}[section]
\newtheorem{lemma}[theorem]{Lemma}
\newtheorem{proposition}[theorem]{Proposition} 
\newtheorem{corollary}[theorem]{Corollary}

\theoremstyle{definition}

\begin{document}

\author{Riccardo Molinarolo\thanks{Dipartimento per lo Sviluppo Sostenibile e la Transizione Ecologica, Universit\`a degli Studi del Piemonte Orientale A. Avogadro, Piazza Sant'Eusebio 5, 13100, Vercelli, Italy. Email: riccardo.molinarolo@uniupo.it}}

\title{Existence result for a nonlinear mixed boundary value problem for the heat equation}

\date{\today}

\maketitle

\noindent
{\bf Abstract:} 
In this paper we study the existence of solutions in parabolic Schauder space of a nonlinear mixed boundary value problem for the heat equation in a perforated domain. From a given regular open set $\Omega\subseteq\mathbb{R}^n$ we remove a cavity $\omega\subseteq \Omega$. On the exterior boundary of $\Omega\setminus\overline{\omega}$ we prescribe a Neumann boundary condition, while on the interior boundary we set a nonlinear Robin-type condition. Under suitable assumptions on the data and by means of Leray Schauder Fixed-Point Theorem, we prove the existence of (at least) one solution $u \in C_{0}^{\frac{1+\alpha}{2}; 1+\alpha}([0,T] \times (\overline{\Omega} \setminus \omega))$.

\vspace{9pt}

\vspace{9pt}


\section{Introduction} 
Boundary value problems for the heat equation have been extensively studied in literature by several authors and are of great interest in various applications. Without any hope of completeness, we mention the classical monograph of O.A.~Lady\v{z}enskaja,  V.A.~Solonnikov, and N.N.~Ural'ceva \cite{LaSoUr68} and the books of Friedmann \cite{Fr08} and Lieberman \cite{Li96}. 

In particular, the mixed Neumann-Robin boundary value problem for the heat equation that we address in this paper is inspired by the recent works by Bacchelli, Di Cristo, Sincich, and Vessella \cite{BaDiSiVe14} and by Nakamura and Wang \cite{NaWa17}. Those authors considered the task of reconstructing unknown inclusions inside a heat conductor from boundary measurement. Parabolic equations with mixed boundary conditions typically arise in the context of inverse problems: for a comprehensive motivation we refer to \cite{BiCeFaIn10,DiRoVe06}. In general, the goal is to analyse situations where a corrosion might occur on part of a boundary and the aim is to recover information on this damaged part that cannot be directly inspected. Moreover, the stability issue (i.e., the continuous dependence of the solutions from the boundary data) and the study of the internal structure of a heat
conductor, such as size, location and shape of anomalies, are widely considered. We also mention that the stationary case, i.e., elliptic equations, has been studied in the theory of inverse problem. Here, we refer to Bacchelli \cite{Ba09}, Inglese \cite{In97} and the reference therein. 

Linear boundary conditions are typically considered in this context. To tackle nonlinear boundary conditions one can use potential theory, a tool that has proven to be very powerful in studying boundary value problems. Furthermore, regarding some questions arising in the theory of inverse problems, in many instances the study not only of the existence of a solution of a boundary value problem, but also its regularity, uniqueness and its dependence upon the perturbation of the domain of definition have been extensively investigated in the literature. Here, we refer to the monographs on ``shape optimization'' by Henrot and Pierre \cite{HePi18}, Novotny and Soko\l owski \cite{NoSo13}, and Soko\l owski and Zol\'esio \cite{SoZo92}.

Moreover, regarding shape perturbation problems, we also mention the method developed by Lanza de Cristoforis and collaborators, the ``functional analytic approach'', which aims to prove differentiability properties of the so-called ``domain-to-solution'' map.  In this context, for elliptic differential operators, we refer to \cite{DaLa10,  LaMu11, LaRo04}, while for the parabolic case we mention \cite{DaLu23, LaLu17,LaLu19}. Finally, also transmission type of boundary conditions and space-periodic geometrical setting have been studied, see \cite{DaLuMoMu24,DaMoMu19,DaMoMu21,Mo19,LuMu18,LuMu20} and reference therein.

The problem we address in this paper stems from \cite{DaLuMoMu24.2}, where the authors have successfully used the functional analytic approach to study a perturbed version of the problem considered in this work (see \eqref{princeq} below), when the shape of the cavity is deformed by a suitable regular diffeomorphism. Assuming the existence of a solution when the cavity is fixed (i.e., the diffeomorphism is the identity map on a reference domain), they have proved that the solution continues to exists for diffeomorphisms ``closed'' to the identity map. Moreover, the ``domain-to-solution'' map is of class $C^\infty$.

Therefore, the present paper aims to establish a direct and self-contained existence result, by means of Leray Schauder Fixed-Point Theorem, for a nonlinear mixed boundary value problem for the heat equation in a perforated domain. In particular, this result justifies the assumption of the existence of a solution for the unperturbed problem analysed in \cite{DaLuMoMu24.2}.

In order to introduce the geometric framework of our problem, we fix once and for all 
\begin{equation*}
    n \in \mathbb{N} \setminus \{0,1\},
\end{equation*}
that will be the dimension of the Euclidean space $\mathbb{R}^n$. Then, we fix a regularity parameter and a final time, namely
\begin{equation*}
    \alpha \in ]0,1[ \text{ and } T>0,
\end{equation*}
and we take two sets $\Omega$ and $\omega$ that satisfy the following conditions:
\begin{equation*}
	\begin{split}
		&\mbox{$\Omega$, $\omega$ are bounded open connected subsets of $\mathbb{R}^n$ of class $C^{1,\alpha}$,} 
		\\
		&\mbox{with connected exteriors  $\mathbb{R}^n\setminus \overline{\Omega}$ and $\mathbb{R}^n\setminus \overline{\omega}$ and $\overline{\omega}\subseteq\Omega$}.
	\end{split}
\end{equation*}
We consider a mixed boundary value problem for the heat equation in a perforated domain, obtained by removing from the given domain $\Omega$ a cavity $\omega$. Therefore, to define the boundary conditions, we fix two functions
\begin{equation*}
\begin{aligned}
f \in C_{0}^{\frac{\alpha}{2}; \alpha}([0,T] \times \partial\Omega), G \in C^{0}([0,T] \times \partial\omega \times \mathbb{R}) \text{ with } G(0,x,0)=0 \text{ for every } x \in \partial\omega,
\end{aligned}
\end{equation*}
which will satisfy suitable assumptions described in Section \eqref{sec existence result} (cf. conditions \eqref{condition NG} and \eqref{condition G}). The function $f$ plays the role of the Neumann datum on the outer boundary $\partial \Omega$. Instead, $G$ determines the non linear Robin-type condition on the inner boundary $\partial \omega$. We consider the following nonlinear problem for a function $u \in C_{0}^{\frac{1+\alpha}{2}; 1+\alpha}([0,T] \times (\overline{\Omega} \setminus \omega))$:
\begin{equation}\label{princeq}
    \begin{cases}
	\partial_t u - \Delta u = 0 & \quad\text{in } ]0,T] \times (\Omega \setminus \overline{\omega}), 
	\\
    \nu_{\Omega} \cdot \nabla  u (t,x) = f(t,x) & \quad \forall (t,x) \in [0,T] \times \partial \Omega,
	\\
	\nu_{\omega} \cdot \nabla  u (t,x) = G(t,x,u(t,x)) & \quad \forall (t,x) \in [0,T] \times \partial \omega,
    \\
    u(0,\cdot)=0 & \quad \text{in } \overline{\Omega} \setminus \omega,
    \end{cases}
\end{equation}
where $\nu_{\Omega}$ and $\nu_{\omega}$ denote the outward unit normal vector field to $\partial \Omega$ and to $\partial  \omega$, respectively.

The aim of the paper is to prove an existence result for solution in the parabolic Schauder space  $C_{0}^{\frac{1+\alpha}{2}; 1+\alpha}([0,T] \times (\overline{\Omega} \setminus \omega))$ for problem \eqref{princeq}. The main result is contained in Theorem \ref{thm u_0}. The proof is based on the Leray-Schauder Fixed-Point Theorem: we briefly describe our strategy. We first establish a representation result for caloric functions in $\overline{\Omega}\setminus \omega$ in terms of single layer heat potential with suitable densities (Lemma \ref{lemma rappr}). This result is based on the isomorphism property of the single layer heat potential given by Theorem \ref{thm V iso}, which proof relies on existence and uniqueness results for the exterior Dirichlet and Neumann problems for the heat equation in parabolic Schauder spaces (we provide a self-contained proof of those results via Fredholm Alternative in Theorems \ref{app thm existence ext dir}\,-\,\ref{app thm existence ext neu}).

Then, we introduce a linear version of problem \eqref{princeq} and we study uniqueness result in $C_{0}^{\frac{1+\alpha}{2}; 1+\alpha}([0,T] \times (\overline{\Omega} \setminus \omega))$ for such problem. Along with that, we investigate the mapping properties of an auxiliary boundary integral operator, $\mathcal{J}_\beta$, arising from the integral formulation of that problem (Proposition \ref{prop J_beta}). Then, we derive the integral system formulation of problem \eqref{princeq}: by means of the auxiliary operator $\mathcal{T}_\beta$, such system is written into a fixed-point type equation (Proposition \ref{Tcontcomp}). Under suitable conditions on the function $f$ and $G$, we establish the existence of at least one fixed point (Proposition \ref{prop mu_0}). Then, using the representation result for caloric functions and the mapping properties of layer heat potentials, our main result is proved.

The paper is organised as follows: in Section \ref{s:prel} we recall some standard notation, we introduce parabolic Schauder spaces and we collect some mapping properties on layer heat potentials, as well as existence results for exterior Dirichlet and Neumann problem for the heat equation. In Section \ref{sec existence result} we prove the main result of the paper, Theorem \ref{thm u_0}, on the existence of (at least) one solution in parabolic Schauder space for problem \eqref{princeq}. An Appendix on uniqueness results for Dirichlet and Neumann problems for the heat equation concludes the paper.

\section{Preliminaries}\label{s:prel}
\subsection{Notation and parabolic Schauder spaces}
The symbol $\mathbb{N}$ denotes the set of natural numbers including $0$. Throughout the paper,
\[
n \in \mathbb{N} \setminus \{0, 1\}
\]
denotes the dimension of the Euclidean ambient space $\mathbb{R}^n$. The ball of center $x \in \mathbb{R}^n$ and radius $r\geq 0$ is denoted by $B(x,r)$. The inverse of an
invertible function $f$ is denoted by $f^{(-1)}$, while the reciprocal of function $g$ is denoted by $g^{-1}$.

If $\Omega$ is an open subset of $\mathbb{R}^n$, then $\overline{\Omega}$ denotes
the closure of $\Omega$, $\partial \Omega$ denotes the boundary of $\Omega$ and $\Omega^- := \mathbb{R}^n \setminus \overline{\Omega}$ denotes the exterior of $\Omega$. 

Let $m\in \mathbb{N}$ and $\alpha \in ]0,1[$. For the definition of open subsets of $\mathbb{R}^n$ of class $C^m$
and $C^{m,\alpha}$, and of the Schauder spaces $C^{m,\alpha}(\overline{\Omega})$ and $C^{m,\alpha}(\partial \Omega)$, we refer to Gilbarg and
Trudinger \cite[pp. 52, 94]{GiTr83}. In particular, we use the subscript ``$b$'' to denote the subspace consisting of those functions that are bounded.

Let $\alpha \in ]0,1[$, $T>0$ and $\Omega$ is an open subset of $\mathbb{R}^n$. Then we set 
\begin{equation*}
    C^{\frac{\alpha}{2},\alpha}([0, T] \times \overline{\Omega}) := \Big\{ u \in C^{0}_{b}([0, T] \times \overline{\Omega}) \, \colon \|u\|_{C^{\frac{\alpha}{2};\alpha}([0, T] \times \overline{\Omega})} < +\infty \Big\},
\end{equation*}
where
\begin{align*}
    \|u\|_{C^{\frac{\alpha}{2};\alpha}([0, T] \times \overline{\Omega})} :=& \sup_{[0, T] \times \overline{\Omega}} |u| + \sup_{\substack{t_1,t_2 \in [0,T] \\ t_1 \neq t_2}} \,\sup_{x \in \overline{\Omega}} \frac{|u(t_1,x) - u(t_2,x)|}{|t_1-t_2|^\frac{\alpha}{2}}
    \\
    & + \sup_{t \in [0,T]} \,\sup_{\substack{x_1,x_2 \in \overline{\Omega} \\ x_1 \neq x_2}} \frac{|u(t,x_1) - u(t,x_2)|}{|x_1-x_2|^\alpha} .
\end{align*}
Then we set
\begin{align*}
    C^{\frac{1+\alpha}{2}; 1+\alpha}([0, T] \times \overline{\Omega}) := \Big\{ u \in C^{0}_{b}([0, T] \times \overline{\Omega}) \, \colon \partial_{x_i} u \in C^{0}_b([0, T] \times \overline{\Omega}) \text{ for all } i \in \{1,\dots,n\}, 
    \\
    \text{and such that }
    \|u\|_{C^{\frac{1+\alpha}{2}; 1+\alpha}([0, T] \times \overline{\Omega})}
    < +\infty \Big\},
\end{align*}
where
\begin{align*}
    \|u\|_{C^{\frac{1+\alpha}{2}; 1+\alpha}([0, T] \times \overline{\Omega})} :=& \sup_{[0, T] \times \overline{\Omega}} |u| + \sum_{i=1}^{n} \|\partial_{x_i} u\|_{C^{\frac{\alpha}{2}; \alpha}([0, T] \times \overline{\Omega})} 
    \\
    & + \sup_{\substack{t_1,t_2 \in [0,T] \\ t_1 \neq t_2}} \,\sup_{x \in \overline{\Omega}} \frac{|u(t_1,x) - u(t_2,x)|}{|t_1-t_2|^{\frac{1+\alpha}{2}}} .
\end{align*}
If $\Omega$ is an open subset of $\mathbb{R}^n$ of class $C^{1,\alpha}$, we can use the local parametrization of $\partial \Omega$ to define the spaces
$C^{\frac{j+\alpha}{2}; j+\alpha}([0, T] \times \partial \Omega)$, $j \in \{0,1\}$, in the natural way. More in general, in a similar way we can define the spaces $C^{j,\alpha}(\mathcal{M})$ and $C^{\frac{j+\alpha}{2}; j+\alpha}([0, T] \times \mathcal{M})$, $j \in \{0,1\}$, on a manifold $\mathcal{M}$ of class $C^{j,\alpha}$ imbedded in $\mathbb{R}^n$ (see \cite[Appendix A]{DaLu23}). In essence, a function of class $C^{\frac{j+\alpha}{2}; j+\alpha}$ is $\left(\frac{j+\alpha}{2}\right)$-H\"older continuous in the time variable,
and $(j,\alpha)$-Schauder regular in the space variable. 

Finally, we use the subscript ``$0$'' to denote the subspace consisting of functions that are
zero at $t = 0$. Namely, for $j \in \{0,1\}$, we set
\[
C_0^{\frac{j+\alpha}{2}; j+\alpha}([0, T] \times \overline{\Omega}):= 
\Big\{ u \in C^{\frac{j+\alpha}{2}; j+\alpha}([0, T] \times \overline{\Omega}) \,:\, u(0,x) = 0 \quad \forall x \in \overline{\Omega} \Big\}.
\]
Then the spaces $C_0^{\frac{j+\alpha}{2}; j+\alpha}([0, T] \times \partial \Omega)$ and $C_0^{\frac{j+\alpha}{2}; j+\alpha}([0, T] \times \mathcal{M})$
are similarly defined.

For functions in parabolic Schauder spaces, the partial derivative $D_x$ with respect to the space variable $x$ will be denoted by $\nabla$, while we will mainten the notation $\partial_t$ for the derivative with respect to the time variable $t$.

For a comprehensive introduction to parabolic Schauder spaces we refer the reader to
classical monographs on the field, for example Lady\v{z}enskaja, Solonnikov, and Ural'ceva
\cite[Ch. 1]{LaSoUr68} (see also \cite{LaLu17,LaLu19}).

Finally, we introduce a notation for superposition operators in parabolic Schauder spaces: if $G$ is a function from
$[0,T] \times \partial \omega \times \mathbb{R}$ to $\mathbb{R}$, then we denote by $\mathcal{N}_G$ the nonlinear superposition operator that take a function $u$ from $[0,T]\times \partial\omega$ to $\mathbb{R}$ to the function $\mathcal{N}_G(u)$ defined by
\[
\mathcal{N}_G(u) (t,x) := G(t,x,u(t,x)) \quad \forall (t,x) \times [0,T]\times \partial\omega.
\]
Here the letter `$\mathcal{N}$' stands for `Nemytskii operator'.

\subsection{Layer heat potentials}
In this section we collect some well-known facts on the layer heat potentials. For proofs
and detailed references we refer to Lady\v{z}enskaja, Solonnikov, and Ural'ceva
\cite{LaSoUr68}. 

Let $S_{n} : \mathbb{R}^{1+n} \setminus
\{(0,0)\}\to \mathbb{R}$ be defined by
\[
S_{n}(t,x):=
\left\{
\begin{array}{ll}
\frac{1}{(4\pi t)^{\frac{n}{2}} }e^{-\frac{|x|^{2}}{4t}}&{\mathrm{if}}\ (t,x)\in ]0,+\infty[ \times{\mathbb{R}}^{n}\,, 
\\
0 &{\mathrm{if}}\ (t,x)\in (]-\infty,0]\times{\mathbb{R}}^{n})\setminus\{(0,0)\}.
\end{array}
\right.
\]
It is well known that $S_n$ is the fundamental solution of the heat operator $\partial_t-\Delta$ in $\mathbb{R}^{1+n} \setminus \{(0,0)\}$.

We now introduce the layer heat potentials. Let $\alpha \in ]0,1[$ and $T>0$. Let
\[
\Omega \text{ be an open bounded subset of }\mathbb{R}^n \text{ of class }C^{1,\alpha}.
\]
For a density $\mu \in L^\infty\big([0,T] \times \partial\Omega\big)$, the single layer heat potential is defined as
\begin{equation*} 
    v_{\Omega} [\mu](t,x) := \int_{0}^{t} \int_{\partial \Omega} S_{n}(t-\tau,x-y) \mu(\tau, y)\,d\sigma_y d\tau \quad \forall\,(t,x) \in [0, T] \times \mathbb{R}^n,
\end{equation*}
while the double layer heat potential is defined as
\begin{equation*} 
    w_{\Omega} [\mu](t,x) := \int_{0}^{t} \int_{\partial \Omega} \frac{\partial}{\partial \nu_\Omega(y)} S_{n}(t-\tau,x-y) \mu(\tau, y)\,d\sigma_y d\tau \quad \forall(t,x) \in [0, T] \times \mathbb{R}^n.
\end{equation*}
Moreover, we set
\begin{equation*}
    W_{\partial\Omega} [\mu](t,x) := \int_{0}^{t} \int_{\partial \Omega} \frac{\partial}{\partial \nu_\Omega(y)} S_{n}(t-\tau,x-y) \mu(\tau, y)\,d\sigma_y d\tau \quad \forall(t,x) \in [0, T] \times \partial\Omega,
\end{equation*}
\begin{equation*}
    W^*_{\partial \Omega}[\mu](t,x) := \int_{0}^t\int_{\partial\Omega} 
    \frac{\partial}{\partial \nu_\Omega(x)} S_{n}(t-\tau,x-y) \mu(\tau,y)\,d\sigma_yd\tau \quad \forall\,(t,x) \in [0,T] \times \partial\Omega,
\end{equation*}
and 
\begin{equation*}
    V_{\partial\Omega}[\mu] := v_{\Omega}[\mu]_{|[0,T]\times \partial\Omega}.
\end{equation*}
The map $V_{\partial\Omega}[\mu]$ is the trace of the single layer heat potential on $[0,T]\times \partial\Omega$, while the maps $W_{\partial\Omega}$ and $W^*_{\partial\Omega}[\mu]$ are related to the jump formulas of the double layer heat potential and the normal
derivative of the single layer heat potential on $[0,T]\times \partial\Omega$ respectively (see Theorem \ref{thmdl} (iii) and Theorem \ref{thmsl} (iii)).

Layer heat potentials enjoy properties similar to those of their standard elliptic counterpart. In the following theorem we collect some well known properties for the single layer heat potential. For the proof we refer to \cite{DaLu23,LaLu17,LaLu19}.

\begin{theorem}\label{thmsl}
Let $\alpha \in ]0,1[$ and $T>0$. Let $\Omega$ be a bounded open subset of $\mathbb{R}^n$ of class $C^{1,\alpha}$. Then the following statements hold.
\begin{itemize}

\item[(i)] Let $\mu \in L^\infty([0,T] \times \partial\Omega)$. Then $v_{\Omega}[\mu]$ is continuous and 
$v_{\Omega}[\mu] \in C^\infty(]0,T[ \times (\mathbb{R}^n \setminus \partial\Omega))$. 
Moreover $v_{\Omega}[\mu]$ solves the heat equation 
in $]0,T]\times (\mathbb{R}^n \setminus \partial\Omega)$.

\item[(ii)] Let $v_\Omega^+[\mu]$ and $v_\Omega^-[\mu]$ denote 
the restrictions of $v_\Omega[\mu]$ to $[0,T] \times \overline{\Omega}$ and to $[0,T]\times \overline{\Omega^-}$, respectively. Then, the map from  $C_0^{\frac{\alpha}{2};  \alpha}([0,T] \times \partial\Omega)$ to  $C_{0}^{\frac{1+\alpha}{2}; 1+\alpha}([0,T] \times \overline{\Omega})$ that takes $\mu$ to $v_{\Omega}^+[\mu]$ is linear and continuous. If $R>0$ is such that $\overline{\Omega} \subseteq B(0,R)$, then the map from  $C_0^{\frac{\alpha}{2};  \alpha}([0,T] \times \partial\Omega)$ to  $C_{0}^{\frac{1+\alpha}{2}; 1+\alpha}([0,T] \times (\overline{B(0,R)}\setminus \Omega^-))$ that associates $\mu$ with $v_{\Omega}^-[\mu]$ is also linear and continuous.

\item[(iii)] Let $\mu \in C_0^{\frac{\alpha}{2};  \alpha}([0,T] \times \partial\Omega)$. Then the following jump formulas hold:
\begin{equation*} 
\frac{\partial}{\partial \nu_\Omega}v_{\Omega}^\pm[\mu](t,x)  = \pm \frac{1}{2}\mu(t,x)
 +W_{\partial\Omega}^*[\mu](t,x), \quad \forall (t,x) \in [0,T] \times \partial\Omega.
\end{equation*}
\end{itemize}
\end{theorem}

In the following theorem we collect some well known properties for the double layer heat potential. For the proof of the following result we refer to \cite{LaSoUr68,LaLu19}. 

\begin{theorem}\label{thmdl}
    Let $\alpha \in ]0,1[$ and $T>0$. Let $\Omega$ be a bounded open subset of $\mathbb{R}^n$ of class $C^{1,\alpha}$. Then the following statements hold.
    \begin{itemize}
    \item[(i)] Let $\mu \in L^\infty([0,T] \times \partial\Omega)$. Then $w_{\Omega}[\mu] \in C^\infty((0,T) \times (\mathbb{R}^n \setminus \partial\Omega))$ and $w_{\Omega}[\mu]$ solves the heat equation in $]0,T] \times (\mathbb{R}^n \setminus \partial\Omega)$.

    \item[(ii)] Let $\mu \in C_0^{\frac{\alpha}{2};  \alpha}([0,T] \times \partial\Omega)$. The restriction $w_\Omega[\mu]_{|[0,T]\times \Omega}$ has a unique extension to a continuous function $w^+_\Omega[\mu]$ from $[0,T]\times \Omega$ to $\mathbb{R}$ and the restriction $w_\Omega[\mu]_{|[0,T]\times \Omega^-}$ has a unique extension to a continuous function $w^-_\Omega[\mu]$ from $[0,T]\times \Omega^-$ to $\mathbb{R}$.

    \item[(iii)]  Let $\mu \in C_0^{\frac{\alpha}{2};  \alpha}([0,T] \times \partial\Omega)$. Then the following jump formulas hold:
    \begin{align*}
        &w^{\pm}_{\Omega}[\mu](t,x) = \mp \frac{1}{2} \mu(t,x) + W_{\partial\Omega}[\mu](t,x), & \forall (t,x) \in [0,T]\times\partial\Omega,  
        \\
        &\frac{\partial}{\partial\nu_{\Omega}} w^+_\Omega[\mu](t,x) = \frac{\partial}{\partial\nu_{\Omega}} w^-_\Omega[\mu](t,x), & \forall (t,x) \in [0,T]\times\partial\Omega.
    \end{align*}

    \item[(iv)] The map from $C_0^{\frac{1+\alpha}{2}; 1+\alpha}([0,T] \times \partial\Omega)$ to $C_0^{\frac{1+\alpha}{2}; 1+\alpha}([0,T] \times \overline{\Omega})$ that takes $\mu$ to $w^+[\mu]$ is linear and continuous. If $R>0$ is such that $\overline{\Omega} \subseteq B(0,R)$, then the map from $C_0^{\frac{1+\alpha}{2}; 1+\alpha}([0,T] \times \partial\Omega)$ to $C_0^{\frac{1+\alpha}{2}; 1+\alpha}([0,T] \times (\overline{B(0,R)} \setminus \Omega))$ that takes $\mu$ to $w^-[\mu]_{[0,T] \times (\overline{B(0,R)} \setminus \Omega)}$ is linear and continuous.
    \end{itemize}
\end{theorem}

Then we recall the following well-known consequences of the
Ascoli-Arzel\'a Theorem on the compactness of the embedding of parabolic Schauder spaces.

\begin{proposition}\label{Ascoli Arzela cons prop}
Let $\alpha \in ]0,1[$ and $T>0$. Let $\Omega$ be a bounded open subset of $\mathbb{R}^n$ of class $C^{1,\alpha}$. Then the following statements hold.
\begin{itemize}
    \item[(i)] The embedding of $C_0^{\frac{1+\alpha}{2}; 1+\alpha}([0,T] \times \partial\Omega)$ into $C_0^{\frac{1+\gamma}{2}; 1+\gamma}([0,T] \times \partial\Omega)$ for any $\gamma \in ]0,\alpha[$;
    
    \item[(ii)] The embedding of $C_0^{\frac{1+\alpha}{2}; 1+\alpha}([0,T] \times \partial\Omega)$ into $C_0^{\frac{\alpha}{2}; \alpha}([0,T] \times \partial\Omega)$ is compact;

    \item[(iii)] The embedding of $C_0^{\frac{\alpha}{2}; \alpha}([0,T] \times \partial\Omega)$ into $C_0^{\frac{\gamma}{2}; \gamma}([0,T] \times \partial\Omega)$ is compact for any $\gamma \in [0,\alpha[$;

    \item[(iv)] The embedding of $C_0^{\frac{\gamma}{2}; \gamma}([0,T] \times \partial\Omega)$ into $C_0^{0}([0,T] \times \partial\Omega)$ is compact for any $\gamma \in ]0,\alpha]$;
\end{itemize}
   
\end{proposition}

We now recall some compactness results for the operators $W_{\partial \Omega}$ and $W^*_{\partial \Omega}$. Those results will be crucial in order to prove Theorems \ref{app thm existence ext dir}\,-\,\ref{app thm existence ext neu} below. 

\begin{theorem}\label{thm W and W*}
Let $\alpha \in ]0,1[$, $\gamma \in ]0,\alpha[$ and $T>0$. Let $\Omega$ be a bounded open subset of $\mathbb{R}^n$ of class $C^{1,\alpha}$. Then the following statements hold.
\begin{itemize}
\item[(i)] The operator $W_{\partial\Omega}$ is linear and continuous from $C_0^{\frac{1+\alpha}{2};  1+\alpha}([0,T] \times \partial\Omega)$ into itself;

\item[(ii)] The operator $W_{\partial\Omega}$ is compact from $C_0^{\frac{1+\alpha}{2};  1+\alpha}([0,T] \times \partial\Omega)$ into itself;

\item[(iii)] The operator $W^*_{\partial\Omega}$ is linear and continuous from $C_0^{0}([0,T] \times \partial\Omega)$ to $C_0^{\frac{\gamma}{2}; \gamma}([0,T] \times \partial\Omega)$ and from $C_0^{\frac{\gamma}{2}; \gamma}([0,T] \times \partial\Omega)$ to $C_0^{\frac{\alpha}{2}; \alpha}([0,T] \times \partial\Omega)$;

\item[(iv)] The operator $W^*_{\partial\Omega}$ is compact from $C_0^{0}([0,T] \times \partial\Omega)$ into itself and from $C_0^{\frac{\alpha}{2};  \alpha}([0,T] \times \partial\Omega)$ into itself.
\end{itemize}
\end{theorem}

\begin{proof}
    The proof of points (i) and (ii) follows by \cite[Thm. 4.5 (ii)]{LaLu19} and by Proposition \ref{Ascoli Arzela cons prop} (i). The proof of point (iii) follows by \cite[Thm. 4.6]{LaLu19}. Then, statement (iv) follows combining point (iii) and Proposition \ref{Ascoli Arzela cons prop} (iii)-(iv).
\end{proof}

We are now in the position to prove two results on the existence of solutions in parabolic Schauder spaces of the heat equation in unbounded domains. We start with the following theorem for the exterior Dirichlet problem. We mention that, although those are possibly known results, we provide a direct proof by means of Fredholm Altervative. (For the notion of sub-exponential growth at infinity we refer to the Appendix \ref{app 1}.)

\begin{theorem}\label{app thm existence ext dir}
    Let $\alpha\in ]0,1[$ and $T>0$. Let $\Omega$ be a bounded open subset of $\mathbb{R}^n$ of class $C^{1,\alpha}$. Let $g \in C_0^{\frac{1+\alpha}{2}; 1+\alpha}([0,T] \times \partial\Omega)$. Then the problem 
    \begin{equation}\label{prob dirichlet ext}
	\begin{cases}
	\partial_t u - \Delta u = 0 & \quad \text{in } [0,T] \times \overline{\Omega^-}, 
	\\
    u = g & \quad \mbox{on } [0,T] \times \partial \Omega,
    \\
    u(0,\cdot)= 0 & \quad \mbox{in } \overline{\Omega^-},
	\end{cases}
	\end{equation}
    has a unique solution $u \in C^{\frac{1+\alpha}{2}; 1+\alpha}_0([0,T]\times \overline{\Omega^-})$ with sub-exponential growth at infinity (i.e. $u$ satisfies \eqref{app: eq sub expo}) given by
    \begin{equation*}
        u = w^-_{\Omega}[\mu] \quad\text{in } [0,T]\times \overline{\Omega^-},
    \end{equation*}
    where $\mu$ is the unique solution in $C_0^{\frac{1+\alpha}{2}; 1+\alpha}([0,T] \times \partial\Omega)$ of
    \begin{equation}\label{app eq integ ext dir}
        \left(\frac{1}{2} I + W_{\partial \Omega} \right) [\mu] = g \quad\text{in } [0,T] \times \partial\Omega. 
    \end{equation}
\end{theorem}

\begin{proof}
    First we notice that uniqueness for problem \eqref{prob dirichlet ext} for solutions in $C^{\frac{1+\alpha}{2}; 1+\alpha}_0([0,T]\times \overline{\Omega^-})$ with sub-exponential growth at infinity follows by Corollary \ref{app cor ext dir}. 
    
    We now prove the existence part of the statement. By known properties for the double layer heat potential (see Theorem \ref{thmdl}), $w^-_{\Omega}[\mu]$ solves problem \eqref{prob dirichlet ext} if and only if $\mu \in C_0^{\frac{1+\alpha}{2}; 1+\alpha}([0,T] \times \partial\Omega)$ solves the integral equation \eqref{app eq integ ext dir}.
    By Theorem \ref{thm W and W*} (ii) we know that $W_{\partial \Omega}$ is compact from $C_0^{\frac{1+\alpha}{2}; 1+\alpha}([0,T] \times \partial\Omega)$ into itself. Hence the operator $\left(\frac{1}{2} I + W_{\partial \Omega} \right)$ is a compact perturbation of an isomorphisms and therefore it is a Fredholm operator of index $0$. Accordingly, the Fredholm Alternative implies that in order to prove the unique solvability of the equation \eqref{app eq integ ext dir} in $C_0^{\frac{1+\alpha}{2}; 1+\alpha}([0,T] \times \partial\Omega)$ it suffices to show that the equation
    \begin{equation}\label{app eq integ ext dir homo}
        \left(\frac{1}{2} I + W_{\partial \Omega} \right) [\mu] = 0 \quad\text{on } [0,T] \times \partial\Omega, 
    \end{equation}
    admits $\mu = 0$ as a unique solution in $C_0^{\frac{1+\alpha}{2}; 1+\alpha}([0,T] \times \partial\Omega)$. Hence, let $\mu \in C_0^{\frac{1+\alpha}{2}; 1+\alpha}([0,T] \times \partial\Omega)$ be a solution of \eqref{app eq integ ext dir homo}. Theorem \ref{thmdl} implies that $w^-_{\Omega}[\mu]$ solves
    \begin{equation*}
	\begin{cases}
	\partial_t u - \Delta u = 0 & \quad \text{in } [0,T] \times \overline{\Omega^-}, 
	\\
    u = 0 & \quad \mbox{on } [0,T] \times \partial \Omega,
    \\
    u(0,\cdot)= 0 & \quad \mbox{in } \overline{\Omega^-}.
	\end{cases}
	\end{equation*}
    Since the double layer heat potential has sub-exponential growth for $|x|\to \infty$ (see \eqref{phi2eq}), then Corollary \ref{app cor ext dir} implies that $w^-_{\Omega}[\mu]=0$ in $[0,T] \times \overline{\Omega^-}$. In particular, $\partial_{\nu_{\Omega}} w^-_{\Omega}[\mu]= 0$. By jumps formulas of Theorem \ref{thmdl} (iii), we know that $\partial_{\nu_{\Omega}} w^-_{\Omega}[\mu] = \partial_{\nu_{\Omega}} w^+_{\Omega}[\mu]$, hence we obtain that $w^+_{\Omega}[\mu]$ solves the interior Neumann problem 
    \begin{equation*}
	\begin{cases}
	\partial_t u - \Delta u = 0 & \quad \text{in } [0,T] \times \overline{\Omega}, 
	\\
    \frac{\partial}{\partial \nu_\Omega} u = 0 & \quad \mbox{on } [0,T] \times \partial \Omega,
    \\
    u(0,\cdot)= 0 & \quad \mbox{in } \overline{\Omega},
	\end{cases}
	\end{equation*}
    and thus, by Theorem \ref{app thm uni int dir-neu}, we conclude that $w^+_{\Omega}[\mu]=0$ in $[0,T] \times \overline{\Omega}$. Finally, by jump formulas of Theorem \ref{thmdl} (iii), we get
    \begin{equation*}
        \mu = \left(\frac{1}{2} I + W_{\partial \Omega} \right) [\mu] - \left(-\frac{1}{2} I + W_{\partial \Omega} \right) [\mu] = w^-_{\Omega}[\mu]_{|[0,T] \times \partial\Omega} - w^+_{\Omega}[\mu]_{|[0,T] \times \partial\Omega}  = 0 \quad \text{on } [0,T] \times \partial\Omega.
    \end{equation*}
\end{proof}

In a similar way we can prove the following result for the exterior Neumann problem. (Again, for the notion of sub-exponential growth at infinity we refer to Appendix \ref{app 1}.)

\begin{theorem}\label{app thm existence ext neu}
    Let $\alpha\in ]0,1[$ and $T>0$. Let $\Omega$ be a bounded open subset of $\mathbb{R}^n$ of class $C^{1,\alpha}$. Let $g \in C_0^{\frac{\alpha}{2};  \alpha}([0,T] \times \partial\Omega)$. Then the following exterior Dirichlet problem 
    \begin{equation}\label{prob neumann ext}
	\begin{cases}
	\partial_t u - \Delta u = 0 & \quad \text{in } [0,T] \times \overline{\Omega^-}, 
	\\
    \frac{\partial}{\partial \nu_\Omega} u = g & \quad \mbox{on } [0,T] \times \partial \Omega,
    \\
    u(0,\cdot)= 0 & \quad \mbox{in } \overline{\Omega^-},
	\end{cases}
	\end{equation}
    has a unique solution $u \in C^{\frac{1+\alpha}{2}; 1+\alpha}_0([0,T]\times \overline{\Omega^-})$ with sub-exponential growth at infinity (i.e. $u$ satisfies \eqref{app: eq sub expo}) given by 
    \begin{equation*}
        u = v^-_{\Omega}[\mu] \quad\text{in } [0,T]\times \overline{\Omega^-},
    \end{equation*}
    where $\mu$ is the unique solution in $C_0^{\frac{\alpha}{2};  \alpha}([0,T] \times \partial\Omega)$ of 
    \begin{equation}\label{app eq integ ext neu}
        \left(-\frac{1}{2} I + W^\ast_{\partial \Omega} \right) [\mu] = g \quad\text{in } [0,T] \times \partial\Omega. 
    \end{equation}
\end{theorem}

\begin{proof}
    First we notice that uniqueness for problem \eqref{prob neumann ext} for solutions in $C^{\frac{1+\alpha}{2}; 1+\alpha}_0([0,T]\times \overline{\Omega^-})$ with sub-exponential growth at infinity follows by Theorem \ref{app thm ext neu}.
    
    We now prove the existence part of the statement. By known properties for single layer heat potential (see Theorem \ref{thmsl} (i)--(iii)), $v^-_{\Omega}[\mu]$ solves problem \eqref{prob neumann ext} if and only if $\mu \in C_0^{\frac{\alpha}{2};  \alpha}([0,T] \times \partial\Omega)$ solves the integral equation \eqref{app eq integ ext neu}.
    By Theorem \ref{thm W and W*} (iv), we know that $W^\ast_{\partial \Omega}$ is compact from $C_0^{\frac{\alpha}{2};  \alpha}([0,T] \times \partial\Omega)$ into itself. Hence the operator $\left(-\frac{1}{2} I + W^\ast_{\partial \Omega} \right)$ is a compact perturbation of an isomorphism and therefore it is a Fredholm operator of index $0$. Accordingly, the Fredholm Alternative implies that in order to prove the unique solvability of the equation \eqref{app eq integ ext neu} in $C_0^{\frac{\alpha}{2};  \alpha}([0,T] \times \partial\Omega)$ it suffices to show that the equation
    \begin{equation}\label{app eq integ ext neu homo}
        \left(-\frac{1}{2} I + W^\ast_{\partial \Omega} \right) [\mu] = 0 \quad\text{in } [0,T] \times \partial\Omega 
    \end{equation}
    admits $\mu = 0$ as a unique solution in $C_0^{\frac{\alpha}{2};  \alpha}([0,T] \times \partial\Omega)$. Hence, let $\mu \in C_0^{\frac{\alpha}{2};  \alpha}([0,T] \times \partial\Omega)$ be a solution of \eqref{app eq integ ext neu homo}. Theorem \ref{thmsl} implies that $v^-_{\Omega}[\mu]$ solves
    \begin{equation*}
	\begin{cases}
	\partial_t u - \Delta u = 0 & \quad \text{in } [0,T] \times \overline{\Omega^-}, 
	\\
    \frac{\partial}{\partial \nu_\Omega} u = 0 & \quad \mbox{on } [0,T] \times \partial \Omega,
    \\
    u(0,\cdot)= 0 & \quad \mbox{in } \overline{\Omega^-}.
	\end{cases}
	\end{equation*}
    Since the single layer heat potential has sub-exponential growth for $|x|\to \infty$ (see \eqref{phi1eq}), then Theorem \ref{app thm ext neu} implies that $v^-_{\Omega}[\mu]=0$. Moreover, by Theorem \ref{thmsl}, $v^+_{\Omega}[\mu]$ solves the interior Dirichlet problem
    \begin{equation*}
	\begin{cases}
	\partial_t u - \Delta u = 0 & \quad \text{in } [0,T] \times \overline{\Omega}, 
	\\
    u = 0 & \quad \mbox{on } [0,T] \times \partial \Omega,
    \\
    u(0,\cdot)= 0 & \quad \mbox{in } \overline{\Omega},
	\end{cases}
    \end{equation*}
    and thus, by Theorem \ref{app thm uni int dir-neu}, we conclude that $v^+_{\Omega}[\mu]=0$ in $[0,T] \times \overline{\Omega}$. Finally, by jump formulas of Theorem \ref{thmsl} (iii), we have that
    \begin{equation*}
        \mu = \frac{\partial}{\partial \nu_\Omega}v_{\Omega}^+[\mu] - \frac{\partial}{\partial \nu_\Omega}v_{\Omega}^-[\mu] = 0 \qquad \text{on } [0,T] \times \partial\Omega.
    \end{equation*}
\end{proof}

We now state the main result of this section, Theorem \ref{thm V iso} below, which will provide the existence and uniqueness of the densities of the representation Lemma \ref{lemma rappr} in Section \ref{sec existence result}. The result is inspired by \cite[Thm. 2]{LuMu18}. We mention that another possible way to prove Theorem \ref{thm V iso} below is based on the arguments in Brown \cite[Prop. 6.2 \& Thm 4.18]{Br89} and some regularity results in Schauder spaces presented in \cite{LaLu17,LaLu19}. However, our proof is self-contained and based on Theorems \ref{app thm existence ext dir}\,-\,\ref{app thm existence ext neu} and the results in the Appendix \ref{app 1}, hence for completeness we describe our alternative argument. 

\begin{theorem}\label{thm V iso}
    Let $\alpha \in ]0,1[$ and $T>0$. Let $\Omega$ be a bounded open subset of $\mathbb{R}^n$ of class $C^{1,\alpha}$. Then, the operator $V_{\partial \Omega}$ is an isomorphism from $C_0^{\frac{\alpha}{2}; \alpha}([0,T] \times \partial\Omega)$ to $C_0^{\frac{1+\alpha}{2}; 1+\alpha}([0,T] \times \partial\Omega)$.
\end{theorem}

\begin{proof}
    By Theorem \ref{thmsl} (ii) and by the continuity of the trace operator, we know that $V_{\partial\Omega}$ is a linear and continuous operator from $C_0^{\frac{\alpha}{2}; \alpha}([0,T] \times \partial\Omega)$ to $C_0^{\frac{1+\alpha}{2}; 1+\alpha}([0,T] \times \partial\Omega)$. Accordingly, by the Open Mapping Theorem, it suffices to prove that $V_{\partial \Omega}$ is a bijection. We first prove injectivity. Hence, let $\mu \in C_0^{\frac{\alpha}{2}; \alpha}([0,T] \times \partial\Omega)$ be such that $V_{\partial\Omega}[\mu] = 0$. By the continuity of the single layer heat potential, we know that $v^+_{\Omega}[\mu]_{|[0,T] \times \partial\Omega}= v^-_{\Omega}[\mu]_{|[0,T] \times \partial\Omega}=0$. Thus, the function $v^+_{\Omega}[\mu]$ solves the Dirichlet problem for the heat equation in $[0,T] \times \Omega$ with zero initial condition and with zero Dirichlet boundary condition. The uniqueness of the solution for the classical Dirichlet problem implies that $v^+_{\Omega}[\mu] = 0$ on $[0,T] \times \overline{\Omega}$ (cf. Theorem \ref{app thm uni int dir-neu}).
    
    Furthermore, since $v^-_{\Omega}[\mu]_{|[0,T] \times\partial\Omega}=0$, then the function $v^-_{\Omega}[\mu]$ solves the Dirichlet problem
    \begin{equation*}
	\begin{cases}
	\partial_t u - \Delta u = 0 & \quad \text{in } [0,T] \times \overline{\Omega^-}, 
	\\
    u = 0 & \quad \mbox{on } [0,T] \times \partial \Omega,
    \\
    u(0,\cdot)= 0 & \quad \mbox{in } \overline{\Omega^-}.
	\end{cases}
    \end{equation*}
    Moreover, we notice that $v^-_{\Omega}[\mu]$ satisfies the following inequality:
    \begin{equation*}
    |v^-_{\Omega}[\mu](t,x)| \leq C \,e^{c|x|^2} \qquad \forall (t,x) \in ]0,T[ \times \Omega^-, 
    \end{equation*}
    for two suitable constants $C,c>0$ (see estimate \eqref{phi1eq} in the Appendix \ref{app 1} for details). Hence by Corollary \ref{app cor ext dir}, we conclude that $v^-_{\Omega}[\mu] = 0$ in $[0,T] \times \overline{\Omega^-}$. Finally, the jump formulas in Theorem \ref{thmsl} (iii) imply that
    \begin{equation*}
    \mu = \frac{\partial}{\partial \nu_\Omega}v_{\Omega}^+[\mu] - \frac{\partial}{\partial \nu_\Omega}v_{\Omega}^-[\mu] = 0 \qquad \text{on } [0,T] \times \partial\Omega. 
    \end{equation*}

    Next, we prove surjectivity. Let $\xi \in C_0^{\frac{1+\alpha}{2}; 1+\alpha}([0,T] \times \partial\Omega)$. By Theorem \ref{app thm existence ext dir} on the existence for the Dirichlet problem for the heat equation in $[0,T] \times \Omega^-$, there exists a unique function $u^-_{\xi} \in C_{0}^{\frac{1+\alpha}{2}; 1+\alpha}([0,T] \times \overline{\Omega^-})$ which solves 
    \begin{equation*}
	\begin{cases}
	\partial_t u - \Delta u = 0 & \quad \text{in } [0,T] \times \overline{\Omega^-}, 
	\\
    u = \xi & \quad \mbox{on } [0,T] \times \partial \Omega,
    \\
    u(0,\cdot)= 0 & \quad \mbox{in } \overline{\Omega^-},
	\end{cases}
    \end{equation*}
    and satisfies the following inequality:
    \begin{equation*}
    |u^-_{\xi}(t,x)| \leq C \,e^{c|x|^2} \qquad \forall (t,x) \in ]0,T[ \times \Omega^-, 
    \end{equation*}
    for two suitable constants $C,c>0$. Then, since $\frac{\partial}{\partial \nu_\Omega} u^-_{\xi} \in C_{0}^{\frac{\alpha}{2}; \alpha}([0,T] \times \partial\Omega)$, by Theorem \ref{app thm existence ext neu} there exists a unique $\mu \in C_{0}^{\frac{\alpha}{2}; \alpha}([0,T] \times \partial\Omega)$ such that $v^-_{\Omega}[\mu]$ solves the problem
    \begin{equation*}
	\begin{cases}
	\partial_t u - \Delta u = 0 & \quad \text{in } [0,T] \times \overline{\Omega^-}, 
	\\
    \frac{\partial}{\partial \nu_{\Omega}} u =  \frac{\partial}{\partial \nu_\Omega} u^-_{\xi} & \quad \mbox{on } [0,T] \times \partial \Omega,
    \\
    u(0,\cdot)= 0 & \quad \mbox{in } \overline{\Omega^-}.
	\end{cases}
    \end{equation*}
    By the uniqueness of the solution for the Neumann problem for the heat equation (see Theorem \ref{app thm ext neu}), we have that $v^-_{\Omega}[\mu] = u^-_{\xi}$. In particular,
    \begin{equation*}
    V_{\partial\Omega}[\mu] = v^-_{\Omega}[\mu]_{|[0,T]\times \partial\Omega} = \xi \qquad \text{on } [0,T] \times \partial \Omega,
    \end{equation*}
    and accordingly the statement follows.
\end{proof}

In the following Theorem \ref{thm V}, we recall some mapping properties and compactness results for the operator $V_{\partial\Omega}$.

\begin{theorem}\label{thm V}
Let $\alpha \in ]0,1[$ and $T>0$. Let $\Omega$ be a bounded open subset of $\mathbb{R}^n$ of class $C^{1,\alpha}$. Then the following statements hold.
\begin{itemize}
\item[(i)] The operator $V_{\partial\Omega}$ is linear and continuous from $C_0^{0}([0,T] \times \partial\Omega)$ to $C_0^{\frac{\alpha}{2};  \alpha}([0,T] \times \partial\Omega)$;

\item[(ii)] The operator $V_{\partial\Omega}$ is compact from $C_0^{0}([0,T] \times \partial\Omega)$ into itself and from $C_0^{\frac{\alpha}{2};  \alpha}([0,T] \times \partial\Omega)$ into itself.
\end{itemize}
\end{theorem}

\begin{proof}
    The proof of point (i) for $n \geq 3$ follows by \cite[Prop. 7.1\,-\,7.2]{LaLu17}, see also \cite[Rmk. 3]{LaLu17} (for $n=2$, a straightforward modification of the aforementioned results provides again the needed proof). Then, statement (ii) follows combining point (i) and Proposition \ref{Ascoli Arzela cons prop} (iii), and Theorem \ref{thmsl} (iv) and Proposition \ref{Ascoli Arzela cons prop} (ii). 
\end{proof}

We finally mention a bootstrap regularity result in the spirit of \cite[Lem. 3.3]{DaMi15}. 

\begin{theorem}\label{thm bootstrap}
    Let $\alpha \in ]0,1[$ and $T>0$. Let $\Omega$ be a bounded open subset of $\mathbb{R}^n$ of class $C^{1,\alpha}$. Let $\mu \in C_0^{0}([0,T] \times \partial\Omega)$. If $\left( \pm\frac{1}{2} I + W^\ast_{\partial \Omega} \right) [\mu] \in C_0^{\frac{\alpha}{2};  \alpha}([0,T] \times \partial\Omega)$, then $\mu \in C_0^{\frac{\alpha}{2};  \alpha}([0,T] \times \partial\Omega)$.
\end{theorem}

\begin{proof}
    If $\mu \in C_0^{0}([0,T] \times \partial\Omega)$ then by Theorem \ref{thm W and W*} (iii) we know that $W^\ast_{\partial \Omega}[\mu] \in C_0^{\frac{\gamma}{2};  \gamma}([0,T] \times \partial\Omega)$ for any $\gamma \in ]0,\alpha[$. Then by the trivial relation 
    \begin{equation*}
    \pm\mu = 2 \left(\pm\frac{1}{2} I + W^\ast_{\partial \Omega} \right) [\mu] - 2 W^\ast_{\partial \Omega}[\mu],    
    \end{equation*}
    we deduce that $\mu \in C_0^{\frac{\gamma}{2};  \gamma}([0,T] \times \partial\Omega)$ for any $\gamma \in ]0,\alpha[$. Then, by \cite[Thm. 4.6]{LaLu19}, $W^\ast_{\partial \Omega}[\mu] \in C_0^{\frac{\alpha}{2};  \alpha}([0,T] \times \partial\Omega)$ and the conclusion follows using again the above relation.
\end{proof}

\section{Existence result via Fixed-Point Theorem}\label{sec existence result}
This section is devoted to prove the main result of this paper, namely Theorem \ref{thm u_0} below. We recall that, throughout this section we fix once for all
\begin{equation*}
	\begin{split}
		&\mbox{$\Omega$, $\omega$ bounded open connected subsets of $\mathbb{R}^n$ of class $C^{1,\alpha}$,} 
		\\
		&\mbox{with connected exteriors  $\mathbb{R}^n\setminus \overline{\Omega}$ and $\mathbb{R}^n\setminus \overline{\omega}$ and $\overline{\omega}\subseteq\Omega$}.
	\end{split}
\end{equation*}
Moreover, as mentioned in the Introduction, we fix two functions
\begin{equation}\label{condition on f and G}
f \in C_{0}^{\frac{\alpha}{2}; \alpha}([0,T] \times \partial\Omega), G \in C^{0}([0,T] \times \partial\omega \times \mathbb{R}) \text{ with } G(0,x,0)=0 \text{ for every } x \in \partial\omega.
\end{equation}
We notice that, the assumptions on $G$ imply that the map 
\begin{equation*}
    \mathcal{N}_G \colon C_0^{0}([0,T] \times \partial\omega) \to C_0^{0}([0,T] \times \partial\omega) \text{ is continuous.}    
\end{equation*} 
Moreover, we will assume the following mapping condition on the superposition operators generated by $G$, namely
\begin{equation}\label{condition NG}
\mathcal{N}_{G} \text{ maps } C^{\frac{\alpha}{2};\alpha}([0,T] \times \partial\omega) \text{ into } C^{\frac{\alpha}{2}; \alpha}([0,T] \times \partial\omega).
\end{equation}

By Theorem \ref{thm V iso} and by the uniqueness result provided by Theorem \ref{app thm uni int dir-neu}, we deduce the validity of the following.

\begin{lemma}\label{lemma rappr}
    The map from $C_0^{\frac{\alpha}{2};  \alpha}([0,T] \times \partial\Omega) \times C_0^{\frac{\alpha}{2};  \alpha}([0,T] \times \partial \omega)$ to the space
    \begin{equation*}
        \left\{u\in C_{0}^{\frac{1+\alpha}{2}; 1+\alpha}([0,T] \times (\overline{\Omega} \setminus \omega)) \,:\, \partial_t u  - \Delta u =0 \quad \text{in } ]0,T] \times \Omega \setminus \overline{\omega} \right\}
    \end{equation*}
    that takes a pair $(\mu,\eta)$ to the function $u_{\Omega,\omega}[\mu,\eta]$ defined by
    \begin{equation}\label{U}
    	u_{\Omega,\omega}[\mu,\eta] := (v^+_{\Omega} [\mu] + v^-_{\omega}[\eta])_{| [0,T] \times (\overline{\Omega} \setminus \omega)},
    \end{equation}
    is bijective.
\end{lemma}

In order to represent the boundary condition of a linearization of problem \eqref{princeq}, we introduce a function
\begin{equation}\label{cond beta}
    \beta \in C^{\frac{\alpha}{2};  \alpha}([0,T] \times \partial\omega).
\end{equation}
We now present a uniqueness result of a linear version of problem \eqref{princeq}. The proof follows, for example, by Liebermann \cite[Cor. 5.4]{Li96} (see also Friedman \cite[Lemma 2, p. 146]{Fr08}).

\begin{lemma}\label{beta prob lemma}
    Let $\beta$ be as in \eqref{cond beta}. Then, the unique solution in $C_{0}^{\frac{1+\alpha}{2}; 1+\alpha}([0,T] \times (\overline{\Omega} \setminus \omega))$ of problem
	\begin{equation}\label{beta problem eq}
    \begin{cases}
	\partial_t u - \Delta u = 0 & \quad\text{in } ]0,T] \times (\Omega \setminus \overline{\omega}), 
	\\
    \frac{\partial}{\partial \nu_\Omega} u (t,x) = 0 & \quad \forall (t,x) \in [0,T] \times \partial \Omega, 
	\\
	\frac{\partial}{\partial \nu_\omega} u (t,x) - \beta(t,x) u(t,x) = 0 & \quad \forall (t,x) \in [0,T] \times \partial \omega,
    \\
    u(0,\cdot)=0 & \quad \text{in } \overline{\Omega} \setminus \omega,
    \end{cases}
    \end{equation}
	is $u=0$.
\end{lemma}

In the next proposition, we study an auxiliary boundary operator, $\mathcal{J}_\beta$, which we will exploit in the integral formulation of our problem in order to recast a fixed point equation. In particular, we prove that $\mathcal{J}_\beta$ is an isomorphism in suitable spaces. The two considered frameworks will be important: the first setting will be used in order to apply Leray-Shauder Theorem (see Proposition \ref{prop mu_0} below) and the second setting will be central to deduce that the solution obtained is actually classical (see Theorem \ref{thm u_0}).

\begin{proposition}\label{prop J_beta}
    Let $\beta$ be as in \eqref{cond beta}. Let $\mathcal{J}_\beta = (\mathcal{J}_{\beta,1}, \mathcal{J}_{\beta,2})$ be the map from $C^{0}_0([0,T] \times \partial\Omega) \times C^{0}_0([0,T] \times \partial\omega)$ into itself that takes a pair $(\mu,\eta)$ to the pair $\mathcal{J}_\beta[\mu,\eta]$ defined by
	\begin{equation*}
	\begin{aligned}
	\mathcal{J}_{\beta,1}[\mu,\eta] &:= \left( \frac{1}{2} I + W^\ast_{\partial\Omega} \right) [\mu] + \nu_{\Omega} \cdot \nabla  v^-_{\omega}[\eta]_{|[0,T] \times\partial\Omega},
	\\
	\mathcal{J}_{\beta,2}[\mu,\eta] &:= \left( -\frac{1}{2} I + W^\ast_{\partial\omega} \right) [\eta] +
    \nu_{\omega} \cdot \nabla  v^+_{\Omega}[\mu]_{|[0,T] \times\partial\omega}
    - \beta (v^+_{\Omega}[\mu]_{|[0,T] \times\partial\omega} + V_{\partial\omega}[\eta]).  
	\end{aligned}
	\end{equation*}
	Then the following statements hold.
	\begin{enumerate}
		\item[(i)] $\mathcal{J}_{\beta}$ is a linear isomorphism from $C^0_0([0,T]\times\partial\Omega) \times C^0_0([0,T]\times\partial\omega)$ into itself.
        
        \item[(ii)] $\mathcal{J}_{\beta}$ is a linear isomorphism from $C_0^{\frac{\alpha}{2};  \alpha}([0,T] \times \partial\Omega) \times C_0^{\frac{\alpha}{2};  \alpha}([0,T] \times \partial\omega)$ into itself.
	\end{enumerate}
\end{proposition}

\begin{proof}
We start proving (i). Define $\bar{\mathcal{J}} := (\bar{\mathcal{J}}_1, \bar{\mathcal{J}}_2)$ to be the linear operator from $C^0_0([0,T]\times\partial\Omega) \times C^0_0([0,T]\times\partial\omega)$ into itself such that
\begin{equation*}
    \bar{\mathcal{J}}_1[\mu,\eta] := \frac{1}{2} \mu, \quad \bar{\mathcal{J}}_2[\mu,\eta] := -\frac{1}{2} \eta,
\end{equation*}
for all $(\mu,\eta) \in C^0_0([0,T]\times\partial\Omega) \times C^0_0([0,T]\times\partial\omega)$. Clearly $\bar{\mathcal{J}}$ is a linear homeomorphism.

Moreover, by the
results of \cite[Lemma A.2, Lemma A.3]{DaLu23} on non-autonomous composition operators and on
time-dependent integral operators with non-singular kernels, by Theorem \ref{thm W and W*} (iv) and Theorem \ref{thm V} (ii) on the compactness of the operator $W^\ast_{\partial\Omega}$ from $C^0_0([0,T] \times \partial\Omega)$ into itself and of the operators $W^\ast_{\partial\omega}$ and $V_{\partial\omega}$ from $C^0_0([0,T] \times \partial\omega)$ into itself, and by the bilinearity and continuity of the product from $C_0^{\frac{\alpha}{2};  \alpha}([0,T] \times \partial\omega) \times C^0_0([0,T] \times \partial\omega)$ to $C^0_0([0,T]\times \partial\omega)$, we deduce that the map from $C^0_0([0,T]\times\partial\Omega) \times C^0_0([0,T]\times\partial\omega)$ into itself that takes a pair $(\mu,\eta)$ to the pair $\mathcal{J}^C_{\beta}[\mu,\eta]$ defined by
\begin{equation*}
\begin{aligned} 
	\mathcal{J}^C_{\beta,1}[\mu,\eta] &:= W^\ast_{\partial\Omega}[\mu] + \nu_{\Omega} \cdot \nabla  v^-_{\omega}[\eta]_{|[0,T] \times\partial\Omega} && \text{on } [0,T] \times \partial\Omega,
    \\
    \mathcal{J}^C_{\beta,2}[\mu,\eta] &:=  W^\ast_{\partial\omega}[\eta]
    +
    \nu_{\omega} \cdot \nabla  v^+_{\Omega}[\mu]_{|[0,T] \times\partial\omega}
    - \beta (v^+_{\Omega}[\mu]_{|[0,T] \times\partial\omega} + V_{\partial\omega}[\eta])  && \text{on } [0,T] \times \partial\omega,
\end{aligned} 
\end{equation*}
is compact. Since compact perturbations of
linear homeomorphisms are Fredholm operators of index $0$, we have that $\mathcal{J}_{\beta} = \bar{\mathcal{J}} + \mathcal{J}^C_{\beta}$
is a Fredholm operator of index $0$.  Hence, to prove point $(i)$, it suffices to establish that $\mathcal{J}_\beta$ is injective. Thus, we now assume that $(\mu,\eta) \in C^0_0([0,T]\times\partial\Omega) \times C^0_0([0,T]\times\partial\omega)$ and that
\begin{equation}\label{J_beta=0}
\mathcal{J}_{\beta}[\mu,\eta] = (0,0).
\end{equation}

Then, by the
results of \cite[Lemma A.2, Lemma A.3]{DaLu23} on non-autonomous composition operators and on
time-dependent integral operators with non-singular kernels, we know that $\nu_{\Omega} \cdot \nabla  v^-_{\omega}[\eta]_{|[0,T] \times\partial\Omega} \in C_0^{\frac{\alpha}{2},\alpha}([0,T]\times \partial\Omega)$, 
and	$v^+_{\Omega}[\mu]_{|[0,T] \times\partial\omega}, \, \nu_{\omega} \cdot \nabla  v^+_{\Omega}[\mu]_{|[0,T] \times\partial\omega} \in C_0^{\frac{\alpha}{2},\alpha}([0,T]\times \partial\omega)$. By Theorem \ref{thm V} (i), $V_{\partial\omega}[\eta] \in C_0^{\frac{\alpha}{2},\alpha}([0,T]\times \partial\omega)$. Hence, by \eqref{J_beta=0} and by the membership of $\beta \in C^{\frac{\alpha}{2},\alpha}([0,T]\times \partial\omega)$, we obtain that 
\begin{equation*}
    \left(\frac{1}{2} I + W^\ast_{\partial\Omega} \right) [\mu] \in C_0^{\frac{\alpha}{2},\alpha}([0,T]\times \partial\Omega),
    \left(-\frac{1}{2} I + W^\ast_{\partial\omega} \right) [\eta],\in C_0^{\frac{\alpha}{2},\alpha}([0,T]\times \partial\omega).
\end{equation*}
Then Theorem \ref{thm bootstrap} implies $(\mu,\eta) \in C_0^{\frac{\alpha}{2};  \alpha}([0,T] \times \partial\Omega) \times C_0^{\frac{\alpha}{2};  \alpha}([0,T] \times \partial\omega)$. 

By the jump formulas (cf.~Theorem \ref{thmsl} (iii)), and by \eqref{J_beta=0}, we deduce that the function $u_{\Omega,\omega}[\mu,\eta]$ defined by \eqref{U} is a solution of the boundary value problem \eqref{beta problem eq}.
Then by Lemma \ref{beta prob lemma}, we have that  $u_{\Omega,\omega}[\mu,\eta] = 0$,  which implies $(\mu,\eta)=(0,0)$, by the uniqueness of the representation provided by Lemma \ref{lemma rappr}.

The proof of statement (ii) proceeds as follows.
First we note that $\mathcal{J}_\beta$ is continuous from $C_0^{\frac{\alpha}{2};  \alpha}([0,T] \times \partial\Omega) \times C_0^{\frac{\alpha}{2}; \alpha}([0,T] \times \partial\omega)$ into itself (cf.~\cite[Lemma A.2, Lemma A.3]{DaLu23} on non-autonomous composition operators and on
time-dependent integral operators with non-singular kernels and Theorems \ref{thm W and W*} and \ref{thm V}). Then we observe that if 
\begin{equation*}
\mathcal{J}_\beta[\mu,\eta] \in C_0^{\frac{\alpha}{2};  \alpha}([0,T] \times \partial\Omega) \times C_0^{\frac{\alpha}{2}; \alpha}([0,T] \times \partial\omega)    
\end{equation*}
for some $(\mu,\eta) \in C_0^0([0,T] \times \partial\Omega) \times C_0^0([0,T] \times \partial\omega)$, then $(\mu,\eta) \in C_0^{\frac{\alpha}{2};  \alpha}([0,T] \times \partial\Omega) \times C_0^{\frac{\alpha}{2}; \alpha}([0,T] \times \partial\omega)$ (see also the argument used after \eqref{J_beta=0} to prove that $(\mu,\eta)$ belongs to $C_0^{\frac{\alpha}{2};  \alpha}([0,T] \times \partial\Omega) \times C_0^{\frac{\alpha}{2}; \alpha}([0,T] \times \partial\omega)$). Then, by statement (i) we deduce that $\mathcal{J}_\beta$ is a bijective continuous linear map from $C_0^{\frac{\alpha}{2};  \alpha}([0,T] \times \partial\Omega) \times C_0^{\frac{\alpha}{2}; \alpha}([0,T] \times \partial\omega)$ into itself. By the Open Mapping Theorem it follows that $\mathcal{J}_\beta$ is a linear homeomorphism from $C_0^{\frac{\alpha}{2};  \alpha}([0,T] \times \partial\Omega) \times C_0^{\frac{\alpha}{2}; \alpha}([0,T] \times \partial\omega)$ into itself.
\end{proof}

We are now ready to convert  \eqref{princeq} into a system of integral equations. We remark that the integral system for problem \eqref{princeq} associated to the representation given by Lemma \ref{lemma rappr} is given by \eqref{intsys} below; however, equation \eqref{princintsys} is a suitable version of \eqref{intsys} that can be seen as a fixed point equation for a compact operator between suitable spaces (see the definition of $\mathcal{T}$ in \eqref{T_beta eq} below).

\begin{proposition}\label{propintsys}
    Let $f,G$ be as in \eqref{condition on f and G}. Let $\beta$ be as in \eqref{cond beta}. Let assumption \eqref{condition NG} holds. Let $(\mu,\eta) \in C_0^{\frac{\alpha}{2};  \alpha}([0,T] \times \partial\Omega) \times C_0^{\frac{\alpha}{2}; \alpha}([0,T] \times \partial\omega)$. Let $u_{\Omega,\omega}$ be defined by \eqref{U} and  $\mathcal{J}_\beta$ be as in Proposition \ref{prop J_beta}.
	Then 
    \begin{equation*}
        u_{\Omega,\omega}[\mu,\eta] \in C_{0}^{\frac{1+\alpha}{2}; 1+\alpha}([0,T] \times (\overline{\Omega} \setminus \omega))    
    \end{equation*}
    is a solution of \eqref{princeq} if and only if 
	\begin{equation}\label{princintsys}
	\begin{aligned}
	\begin{pmatrix}
	\mu
	\\
	\eta
	\end{pmatrix}
	= \mathcal{J}_\beta^{(-1)}&
	\left[
	\begin{pmatrix}
	f
	\\
	\mathcal{N}_{G}(v^+_{\Omega}[\mu]_{|[0,T] \times\partial\omega} + V_{\partial\omega}[\eta])
	\end{pmatrix}
	-
    \begin{pmatrix}
    0
    \\
    \beta \left( v^+_{\Omega}[\mu]_{|[0,T] \times\partial\omega} + V_{\partial\omega}[\eta] \right) 
    \end{pmatrix}
    \right].
	\end{aligned}
	\end{equation}
\end{proposition}

\begin{proof}
	By Theorem \ref{thmsl} (iii), we know that if $(\mu,\eta) \in C_0^{\frac{\alpha}{2};  \alpha}([0,T] \times \partial\Omega) \times C_0^{\frac{\alpha}{2}; \alpha}([0,T] \times \partial\omega)$ then the function $u_{\Omega,\omega}[\mu,\eta]$ defined by \eqref{U} is a solution of problem \eqref{princeq} if an only if
	\begin{equation}\label{intsys}
    \begin{pmatrix}
	\left( \frac{1}{2} I + W^\ast_{\partial\Omega} \right) [\mu] + \nu_{\Omega} \cdot \nabla  v^-_{\omega}[\eta]_{|[0,T] \times\partial\Omega}
	\\
	\left( -\frac{1}{2} I + W^\ast_{\partial\omega} \right) [\eta] + \nu_{\omega} \cdot \nabla  v^+_{\Omega}[\mu]_{|[0,T] \times\partial\omega} 
	\end{pmatrix} =
	\begin{pmatrix}
	f
	\\
	\mathcal{N}_{G}(v^+_{\Omega}[\mu]_{|[0,T] \times\partial\omega} + V_{\partial\omega}[\eta])
	\end{pmatrix}.
	\end{equation}
	Then, subtracting in both sides of \eqref{intsys} the term
	\begin{equation*}
	\begin{pmatrix}
    0
    \\
    \beta \left( v^+_{\Omega}[\mu]_{|[0,T] \times\partial\omega} + V_{\partial\omega}[\eta] \right) 
    \end{pmatrix}
	\in C_0^{\frac{\alpha}{2};  \alpha}([0,T] \times \partial\Omega) \times C_0^{\frac{\alpha}{2}; \alpha}([0,T] \times \partial\omega),
	\end{equation*} 
	and using Proposition \ref{prop J_beta} (ii), the validity of the statement follows. 
\end{proof}

We now introduce an auxiliary map. For $f,G$ as in \eqref{condition on f and G}, $\beta$ as in \eqref{cond beta} and $\mathcal{J}_\beta$ as in Proposition \ref{prop J_beta}, we denote by $\mathcal{T}_\beta$ the map from $C_0^0([0,T] \times \partial\Omega) \times C_0^0([0,T] \times \partial\omega)$ into itself defined by
\begin{equation}\label{T_beta eq}
\mathcal{T}_\beta(\mu,\eta) := \mathcal{J}_\beta^{(-1)} 
	\left[
	\begin{pmatrix}
	f
	\\
	\mathcal{N}_{G}(v^+_{\Omega}[\mu]_{|[0,T] \times\partial\omega} + V_{\partial\omega}[\eta])
	\end{pmatrix}
	-
    \begin{pmatrix}
    0
    \\
    \beta \left( v^+_{\Omega}[\mu]_{|[0,T] \times\partial\omega} + V_{\partial\omega}[\eta] \right) 
    \end{pmatrix}
    \right].
\end{equation}

We study the continuity and compactness of $\mathcal{T}_\beta$ in the following proposition. We recall that a nonlinear operator is said to be compact if and only if it is continuous and maps bounded sets into relatively compact sets. 

\begin{proposition}\label{Tcontcomp}
Let $f,G$ be as in \eqref{condition on f and G}. Let $\beta$ be as in \eqref{cond beta}. Let $\mathcal{T}_\beta$ be as in \eqref{T_beta eq}. Then $\mathcal{T}_\beta$ is a continuous (nonlinear) operator from $C_0^0([0,T] \times \partial\Omega) \times C_0^0([0,T] \times \partial\omega)$ into itself and is compact. 
\end{proposition}

\begin{proof}
By the results of \cite[Lemma A.2, Lemma A.3]{DaLu23} on non-autonomous composition operators and on time-dependent integral operators with non-singular kernels and by Proposition \ref{Ascoli Arzela cons prop} (iv) with $\gamma = \alpha$, $v^+_{\Omega}[\cdot]_{|[0,T] \times\partial\omega}$ is compact from $C_0^0([0,T] \times \partial\Omega)$ into $C_0^0([0,T] \times \partial\omega)$. By Theorem \ref{thm V} (ii), $V_{\partial\omega}$ is compact from $C_0^0([0,T] \times \partial\omega)$ into itself. Hence, by the bilinearity and continuity of the product of continuous functions, the map from $C_0^0([0,T] \times \partial\Omega) \times C_0^0([0,T] \times \partial\omega)$ into itself that takes the pair $(\mu,\eta)$ to the pair given by
\begin{equation*}
\begin{pmatrix}
0
\\
\beta \left( v^+_{\Omega}[\mu]_{|[0,T] \times\partial\omega} + V_{\partial\omega}[\eta] \right) 
\end{pmatrix}
\end{equation*} 
is continuous and compact. Moreover, by the assumptions on $G$, the operator $\mathcal{N}_{G}$ is continuous from $C_0^0([0,T] \times \partial\omega)$ into itself. Thus, the map from $C_0^0([0,T] \times \partial\Omega) \times C_0^0([0,T] \times \partial\omega)$ to itself that takes the pair $(\mu,\eta)$ to the pair
\begin{equation*}
\begin{pmatrix}
f
\\
\mathcal{N}_{G}(v^+_{\Omega}[\mu]_{|[0,T] \times\partial\omega} + V_{\partial\omega}[\eta])
\end{pmatrix}
-
\begin{pmatrix}
0
\\
\beta \left( v^+_{\Omega}[\mu]_{|[0,T] \times\partial\omega} + V_{\partial\omega}[\eta] \right) 
\end{pmatrix}
\end{equation*}
is continuous and compact. Finally, by Proposition \ref{prop J_beta} (i),  $\mathcal{J}_\beta$ is a linear isomorphism from $C^0_0([0,T]\times\partial\Omega) \times C^0_0([0,T]\times\partial\omega)$, hence the statement follows.
\end{proof}

In what follows we will assume the following growth condition on the superposition operator $\mathcal{N}_G$ with respect to the function $\beta$ (cf. \eqref{cond beta}), namely there exists a positive constant $C_G>0$ and an exponent $\delta \in ]0,1[$ such that 
\begin{equation}\label{condition G}
\left\| \mathcal{N}_{G}(h) - \beta h \right\|_{C^{0}([0,T] \times \partial\omega)} \leq C_G \left(1+\|h\|_{C^{0}([0,T] \times \partial\omega)} \right)^\delta \quad \forall h \in C^{0}([0,T] \times \partial\omega).
\end{equation}
In Proposition \ref{prop mu_0} below, we exploit the above condition and we prove the existence of a solution of the system \eqref{princintsys} in the Banach space
\begin{equation*}
    C_0^{0}([0,T] \times \partial\Omega) \times C_0^{0}([0,T] \times \partial\omega).   
\end{equation*}
Our argument is based on the Leray-Schauder Fixed-Point Theorem (cf.~Gilbarg and Trudinger \cite[Thm.~11.3]{GiTr83}). For the reader's convenience we present its statement below.

\begin{theorem}[Leray-Schauder Fixed-Point Theorem]\label{Thm Leray Schauder}
Let $\mathcal{X}$ be a Banach space. Let $\mathcal{T}$ be a continuous operator from $\mathcal{X}$ into itself. If $\mathcal{T}$ is compact and there exists a constant $M \in ]0,+\infty[$ such that $\|x\|_{\mathcal{X}} \leq M$ for all $(x,\lambda) \in \mathcal{X} \times [0,1]$ satisfying $x=\lambda \mathcal{T}(x)$, then $\mathcal{T}$ has at least one fixed point $x \in \mathcal{X}$ such that $\|x\|_{\mathcal{X}} \leq M$.
\end{theorem}

We are now in the position to prove the existence of a fixed point for equation \eqref{princintsys}. In fact, by Proposition \ref{Tcontcomp} and Theorem \ref{Thm Leray Schauder} we deduce the following.

\begin{proposition}\label{prop mu_0}
Let $f,G$ be as in \eqref{condition on f and G}. Let $\beta$ be as in \eqref{cond beta}. Let assumption \eqref{condition G} holds. Let $\mathcal{J}_\beta$ be as in Proposition \ref{prop J_beta}. Then the nonlinear system \eqref{princintsys} has at least one solution $(\mu_0,\eta_0) \in C_0^{0}([0,T] \times \partial\Omega) \times C_0^{0}([0,T] \times \partial\omega)$. 
\end{proposition}

\begin{proof}
    Let $\mathcal{X} := C_0^{0}([0,T] \times \partial\Omega) \times C_0^{0}([0,T] \times \partial\omega)$. By Proposition \ref{Tcontcomp}, $\mathcal{T}_\beta$ is a continuous and compact operator from $\mathcal{X}$. So in order to apply the Leray-Schuder Theorem \ref{Thm Leray Schauder}, we are left to show that if $\lambda \in [0,1]$ and if
    \begin{equation}\label{mu=lamdaT(mu)}
	(\mu,\eta) = \lambda \mathcal{T}_\beta(\mu,\eta)
	\end{equation}
	with $(\mu,\eta) \in \mathcal{X}$, then there exists a constant $C > 0$ (which does not depend on $(\mu,\eta)$ and $\lambda$), such that
    \begin{equation}\label{mu<C}
	\|(\mu,\eta)\|_{\mathcal{X}} \leq C.
    \end{equation}
    By \eqref{mu=lamdaT(mu)} and by $|\lambda| \leq 1$, we clearly obtain that
    \begin{equation}\label{|mu|<|T_beta(mu)|}
      \|(\mu,\eta)\|_{\mathcal{X}} \leq \|\mathcal{T}_\beta(\mu,\eta)\|_{\mathcal{X}}\,.
    \end{equation}	
    Hence, by \eqref{|mu|<|T_beta(mu)|} and by the definition of $\mathcal{T}_\beta$ in \eqref{T_beta eq}, we deduce that there exist two constants $C_1,C_2 > 0$, which depend only on the operator norm of $\mathcal{J}_\beta^{(-1)}$ from $C_0^{0}([0,T] \times \partial\Omega) \times C_0^{0}([0,T] \times \partial\omega)$ (cf.~Proposition \ref{prop J_beta} (i)), on $\|f\|_{C_0^{0}([0,T] \times \partial\Omega)}$, on the constant $C_G > 0$ provided by the growth condition \eqref{condition G}, on the norm of the bounded linear operator $v^+_{\Omega}[\cdot]_{|[0,T] \times \partial\omega}$ from $C_0^0([0,T]\times \partial\Omega)$ to $C_0^0([0,T] \times \partial\omega)$, and on the norm of the bounded linear operator $V_{\partial\omega}$ from $C_0^0([0,T] \times \partial\omega)$ into itself (cf.~Theorem \ref{thm V} (ii)), such that
    \begin{equation*}
        \|(\mu,\eta)\|_{\mathcal{X}} \leq C_1 \left(1 + C_2 \|(\mu,\eta)\|_{\mathcal{X}} \right)^\delta\, .
    \end{equation*}
    Then, by a straightforward calculation and by considering separately the cases $\|(\mu,\eta)\|_{\mathcal{X}} < 1$ and $\|(\mu,\eta)\|_{\mathcal{X}} \geq 1$, we can show the existence of a constant $C>0$ such that inequality \eqref{mu<C} holds true (cf. Lanza de Cristoforis \cite[proof of Thm.~7.2]{La07}). By Theorem \ref{Thm Leray Schauder} there exists at least one solution $(\mu_0,\eta_0) \in \mathcal{X}$ of  $(\mu,\eta) = \mathcal{T}_\beta(\mu,\eta)$.	By the definition of $\mathcal{T}_\beta$ (cf. \eqref{T_beta eq}), we conclude that $(\mu_0,\eta_0)$ is a solution in $\mathcal{X}$ of the nonlinear system \eqref{princintsys}.
\end{proof}

We are now ready to prove a regularity result for the fixed point provided by Proposition \ref{prop mu_0}, and, thus, an existence result for  problem \eqref{princeq}. We state our main result.

\begin{theorem}\label{thm u_0}
Let $f,G$ be as in \eqref{condition on f and G}. Let $\beta$ be as in \eqref{cond beta}. Let assumptions \eqref{condition NG} and \eqref{condition G} hold. Then the nonlinear system \eqref{princintsys} has at least one solution
\begin{equation*}
    (\mu_0,\eta_0) \in C_0^{\frac{\alpha}{2};  \alpha}([0,T] \times \partial\Omega) \times C_0^{\frac{\alpha}{2};  \alpha}([0,T] \times \partial\omega).
\end{equation*}
In particular, problem \eqref{princeq} has at least one solution $u_0 \in C_{0}^{\frac{1+\alpha}{2}; 1+\alpha}([0,T] \times (\overline{\Omega} \setminus \omega))$ given by 
\begin{equation}\label{u_0}
    u_0 :=  u_{\Omega, \omega}[\mu_0,\eta_0],
\end{equation}
where the function $u_{\Omega,\omega}[\mu_0,\eta_0]$ is defined by \eqref{U}.
\end{theorem}

\begin{proof}
By Proposition \ref{prop mu_0}, there exists $(\mu_0,\eta_0) \in C_0^{0}([0,T] \times \partial\Omega) \times C_0^{0}([0,T] \times \partial\omega)$ such that $
(\mu_0,\eta_0) = \mathcal{T}_\beta(\mu_0,\eta_0)$ (cf.~\eqref{T_beta eq}). 
By the results of \cite[Lemma A.2, Lemma A.3]{DaLu23} on non-autonomous composition operators and on time-dependent integral operators with non-singular kernels, $v^+_{\Omega}[\mu_0]_{|[0,T] \times\partial\omega}$ belongs to $C_0^{\frac{\alpha}{2};  \alpha}([0,T] \times \partial\omega)$. By Theorem \ref{thm V} (i), $V_{\partial \omega}[\eta_0]$ belongs to $C_0^{\frac{\alpha}{2};  \alpha}([0,T] \times \partial\omega)$. By the membership of $\beta \in C^{\frac{\alpha}{2};  \alpha}([0,T] \times \partial\omega)$ given by \eqref{cond beta}, we deduce that 
\begin{equation*}
    \beta \left(v^+_{\Omega}[\mu_0]_{|[0,T] \times\partial\omega} + V_{\partial\omega}[\eta_0]\right) \in C_0^{\frac{\alpha}{2};  \alpha}([0,T] \times \partial\omega).
\end{equation*}
Then, by assumption \eqref{condition NG} and by the membership of $f \in C_0^{\frac{\alpha}{2}; \alpha}([0,T] \times \partial\Omega)$, we get that
\begin{equation*}
	\begin{pmatrix}
	f
	\\
	\mathcal{N}_{G}(v^+_{\Omega}[\mu]_{|[0,T] \times\partial\omega} + V_{\partial\omega}[\eta])
	\end{pmatrix}
	-
    \begin{pmatrix}
    0
    \\
    \beta \left( v^+_{\Omega}[\mu]_{|[0,T] \times\partial\omega} + V_{\partial\omega}[\eta] \right) 
    \end{pmatrix}
\end{equation*}
belongs to the product space $C_0^{\frac{\alpha}{2}; \alpha}([0,T] \times \partial\Omega) \times C_0^{\frac{\alpha}{2}; \alpha}([0,T] \times \partial\omega)$. Finally, by Proposition \ref{prop J_beta} (ii), we obtain that 
\begin{equation*}
(\mu_0,\eta_0) \in C_0^{\frac{\alpha}{2};  \alpha}([0,T] \times \partial\Omega) \times C_0^{\frac{\alpha}{2};  \alpha}([0,T] \times \partial\omega).    
\end{equation*}
In particular, by Proposition \ref{propintsys} we deduce that the function given by \eqref{u_0} is a solution of \eqref{princeq}.
\end{proof}

\appendix
\section{Appendix}
In this Appendix we collect some facts on the uniqueness for Dirichlet and Neumann problems for the heat equation. The results presented here are probably known to the reader, but we recall them for sake of completeness and for setting them in suitable parabolic Schuader spaces.

\subsection{Uniqueness results for exterior Dirichlet and Neumann problems}\label{app 1}

Let $\alpha \in ]0,1[$ and $T>0$. Let $\Omega$ be a bounded open subset of $\mathbb{R}^n$ of class $C^{1,\alpha}$. A function $u \in C^{0}([0,T]\times \overline{\Omega^-})$ has sub-exponential growth at infinity if there exist two constants $C,c > 0$ such that 
\begin{equation}\label{app: eq sub expo}
    u(t,x) \leq C \,e^{c|x|^2} \qquad \forall (t,x) \in ]0,T[ \times \Omega^-.
\end{equation}

We have the following maximum principle for the heat equation in exterior domains, which is a straightforward adaptation of a result of Evans \cite[Sec 2.3.3, Thm. 6]{Ev2022}. 

\begin{theorem}
    Let $\alpha \in ]0,1[$ and $T>0$. Let $\Omega$ be a bounded open subset of $\mathbb{R}^n$ of class $C^{1,\alpha}$. Let $g \in C^{0}_0([0,T]\times \partial\Omega)$ and $ f\in C^{0}_b(\overline{\Omega^-})$. Let $u \in C^{0}_0([0,T]\times \overline{\Omega^-})$ be a solution of the following exterior Dirichlet problem
    \begin{equation*}
	\begin{cases}
	\partial_t u - \Delta u = 0 & \quad \text{in } [0,T] \times \Omega^-, 
	\\
    u = g & \quad \mbox{on } [0,T] \times \partial \Omega,
    \\
    u(0,\cdot)= f & \quad \mbox{in } \overline{\Omega^-}.
	\end{cases}
	\end{equation*}
    Moreover, assume that $u$ has sub-exponential growth at infinity (i.e., $u$ satisfies \eqref{app: eq sub expo}). Then 
    \begin{equation*}
    u(t,x) \leq \max \left\{ \max_{[0,T] \times \partial\Omega} g \, , \, \max_{ \overline{\Omega^-}} f \right\} \qquad \forall (t,x) \in [0,T] \times \overline{\Omega^-}.
    \end{equation*}
\end{theorem}

By the previous theorem applied to $u$ and to $-u$ we can deduce the following uniqueness theorem for
the exterior Dirichlet problem.

\begin{corollary}\label{app cor ext dir}
    Let $\alpha \in ]0,1[$ and $T>0$. Let $\Omega$ be a bounded open subset of $\mathbb{R}^n$ of class $C^{1,\alpha}$. Let $u \in C^{0}_0([0,T]\times \overline{\Omega^-})$ be a solution of the following exterior Dirichlet problem
    \begin{equation*}
	\begin{cases}
	\partial_t u - \Delta u = 0 & \quad \text{in } [0,T] \times \Omega^-, 
	\\
    u = 0 & \quad \mbox{on } [0,T] \times \partial \Omega,
    \\
    u(0,\cdot)= 0 & \quad \mbox{in } \overline{\Omega^-}.
	\end{cases}
	\end{equation*}
    Moreover, assume that $u$ has sub-exponential growth at infinity (i.e., $u$ satisfies \eqref{app: eq sub expo}).
    Then 
    \begin{equation*}
        u=0 \quad \text{in } [0,T] \times \overline{\Omega^-}.
    \end{equation*}
\end{corollary}

In order to prove a uniqueness result for the exterior Neumann problem for the heat equation we will need the following lemma on interior estimates for the heat equation (see, e.g., Evans \cite[Thm. 8, p. 59]{Ev2022}). First we introduce a notation: for every $(t,x) \in \mathbb{R}^{1+n}$, let us define
\begin{equation*}
    C(t,x,r) := \left\{(s,y)\in \mathbb{R}^{1+n} \,: \, t-r^2 \leq s \leq t, \, |x-y| \leq r \right\} = B(x,r) \times [t - r^2, t],
\end{equation*}
denoting the closed circular cylinder of radius $r$, height $r^2$ and top point $(t,x)$.
\begin{lemma}[Interior estimates] \label{intest}
    Let $\Omega$ be an open subset of $\mathbb{R}^n$ and let $T>0$. Let $u$ be a solution of the heat equation in $]0, T[ \times \Omega$. Then, for all $k \in \mathbb{N}$, $\eta \in \mathbb{N}^n$ there exists a constant $C_{k,\eta}$ such that
    \begin{equation*}
        |\partial^k_t D^\eta_x u(t,x)| \leq C_{k,\eta} r^{-2k-|\eta|-n-2} \int_{t-r^2}^{t} \int_{B(x,r)} |u(\tau,y)| \,dy\,d\tau
    \end{equation*}
    for all $r > 0$ and for all $(t,x) \in ]0,T[ \times U$ such that $C(t,x,r) \subseteq ]0,T[ \times \Omega$.
\end{lemma}

Moreover we will also need the following parabolic analog of the Third Green Identity. The following lemma is a simple application of Divergence Theorem and classical properties of the fundamental solution of the heat equation $S_n$ (see e.g. \cite[Eq. 10, p. 47 \& Eq. 16, p. 51]{Ev2022}).

\begin{lemma}[Third Green Identity] \label{thgr}
    Let $\Omega$ be a bounded open subset of $\mathbb{R}^n$ of class $C^1$. Let $a,b \in \mathbb{R}$ such that $a<b$. Let $f \in C^0([a,b] \times \overline{\Omega})$ be one time continuously differentiable with respect to the time variable and two time continuously differentiable with respect to the space variable. Then
    \begin{equation*}
    \begin{split}
        &f(t,x) = \int_{\Omega} S_n (t-a,x-y) f(a,y) \,dy + \int_{a}^{t} \int_{\Omega} S_n (t-\tau,x-y) \left( \partial_t f(\tau,y) -\Delta f(\tau,y) \right) \, dy \, d\tau
        \\
        &+ \int_{a}^{t} \int_{\partial \Omega} S_n (t-\tau,x-y) \, \frac{\partial}{\partial \nu_\Omega(y)} f(\tau,y) \, d\sigma_y \, d\tau
        + \int_{a}^{t} \int_{\partial \Omega} \frac{\partial}{\partial \nu_\Omega(y)} S_n (t-\tau,x-y) f(\tau,y) \, d\sigma_y \, d\tau,
    \end{split}
    \end{equation*}
    for all $(t,x)\in ]a,b[ \times \Omega$.
\end{lemma}

Next, we are ready to state and prove our uniqueness result for the exterior Neumann problem. We do this by means of the following theorem.

\begin{theorem}\label{app thm ext neu}
    Let $\alpha \in ]0,1[$ and $T>0$. Let $\Omega$ be a bounded open subset of $\mathbb{R}^n$ of class $C^{1,\alpha}$. Let $u \in C^{\frac{1+\alpha}{2}; 1+\alpha}_0([0,T]\times \overline{\Omega^-})$ be a solution of the following exterior Neumann problem
    \begin{equation*}
	\begin{cases}
	\partial_t u - \Delta u = 0 & \quad \text{in } [0,T] \times \overline{\Omega^-}, 
	\\
    \frac{\partial}{\partial \nu_\Omega} u = 0 & \quad \mbox{on } [0,T] \times \partial \Omega,
    \\
    u(0,\cdot)= 0 & \quad \mbox{in } \overline{\Omega^-}.
	\end{cases}
	\end{equation*}
    Moreover, assume that $u$ has sub-exponential growth at infinity (i.e., $u$ satisfies \eqref{app: eq sub expo}).
    Then 
    \begin{equation*}
        u=0 \quad \text{in } [0,T] \times \overline{\Omega^-}.
    \end{equation*}
\end{theorem}

\begin{proof}
    The strategy of the proof will be the following. First of all we want to prove that under our assumption, an integral representation formula for $u$ holds. With such a representation formula at hand, we can prove that actually $u$ enjoys an exponential decay at infinity in a small interval in time $[0, t_0]$. Then, since $u$ decays fast enough at infinity, we can use the Divergence Theorem inside an energy argument in order to prove our uniqueness result in $[0, t_0]$.  Finally, we can repeat the argument in $[t_0, 2t_0]$, $[2t_0, 3t_0], \dots$, and conclude the proof.
    
    We note that if $t_0 > 0$, then there exist two constants $K_1, K_2 > 0$ such that
    \begin{align}
    \label{phi1eq}
    &|S_n(\tau,z)| \leq K_1 e^{-\frac{|z|^2}{8t_0}} \qquad\forall (\tau,z) \in [0,t_0] \times B(0,1)^-,
    \\
    \label{phi2eq}
    &|\nabla S_n(\tau,z)| \leq K_2 e^{-\frac{|z|^2}{8t_0}} \qquad \forall (\tau,z) \in [0,t_0] \times B(0,1)^-.
    \end{align}
    Indeed, 
    \begin{equation*}
        S_n(\tau,z) = \frac{1}{(4\pi \tau)^{\frac{n}{2}}} e^{-\frac{|z|^2}{4 \tau}} = \frac{1}{(4\pi \tau)^{\frac{n}{2}}} e^{-\frac{|z|^2}{8 \tau}} e^{-\frac{|z|^2}{8 \tau}} \qquad \forall (\tau,z) \in [0,t_0] \times B(0,1)^-.
    \end{equation*}
    Moreover the term $\dfrac{1}{(4\pi \tau)^{\frac{n}{2}}} e^{-\frac{|z|^2}{8 \tau}}$ is bounded for $(\tau,z) \in [0,t_0] \times B(0,1)^-$ and clearly
    \begin{equation*}
        e^{-\frac{|z|^2}{8 \tau}} \leq e^{-\frac{|z|^2}{8 t_0}} \qquad \forall (\tau,z) \in [0,t_0] \times B(0,1)^-,
    \end{equation*}
    which completes the proof of \eqref{phi1eq}. The inequality \eqref{phi2eq} can be proved by arguing in a similar way. Next we take $\xi \in C^{\infty}(\mathbb{R}^n)$ such that $\xi = 1$ in a neighbourhood of infinity and $\xi = 0$ in a neighbourhood of $\partial\Omega$. Then the Third Green Identity of Lemma \ref{thgr} applied to the case $U = B(0, R) \setminus \overline{\Omega}$, $a = 0$, $b = T$, and $f = \xi u$, with $R > 0$ enough big so that $\xi = 1$ in $\mathbb{R}^n \setminus B(0, R)$, implies that
    \begin{equation}\label{repf1}
    \begin{aligned}
        \xi(x)u(t,x) =& \int_{0}^{t} \int_{\partial B(0,R)} S_n (t-\tau,x-y) \, \frac{\partial}{\partial \nu_{B(0,R)}(y)} u(\tau,y) \, d\sigma_y \, d\tau
        \\
        &+ \int_{0}^{t} \int_{\partial B(0,R)} \frac{\partial}{\partial \nu_{B(0,R)}(y)} S_n (t-\tau,x-y) u(\tau,y) \, d\sigma_y \, d\tau
        \\
        &+ \int_{0}^{t} \int_{B(0,R) \setminus \overline{\Omega}} S_n (t-\tau,x-y) \left( 2 \nabla u(\tau,y) \cdot \nabla \xi(y) + u(\tau,y) \Delta \xi(y) \right) \, dy \, d\tau,
    \end{aligned}
    \end{equation}
    for all $(t, x) \in ]0, T[\times (B(0, R) \setminus \overline{\Omega})$. Next, we plan to send $R$ to infinity and analyze the behavior of the right hand side of equality \eqref{repf1}. We consider the second term in the right hand side of \eqref{repf1}. Let $t_0$ be small enough so that
    \begin{equation*}
        c < \frac{1}{8 t_0}.
    \end{equation*}
    We fix $t \in ]0, t_0]$, $x \in \Omega^-$. Our growth assumption on $u$ and \eqref{phi2eq} imply that
    \begin{align*}
        \left| \frac{\partial}{\partial \nu_{B(0,R)}(y)} S_n (t-\tau,x-y) u(\tau,y)  \right| &\leq C K_2 e^{c|y|^2} e^{-\frac{|x-y|^2}{8t_0}}
        \leq C K_2 e^{cR^2} e^{-\frac{(|x|-R)^2}{8t_0}}
        \\
        & \leq C K_2 e^{\left(c-\frac{1}{8t_0}\right)R^2} e^{-\frac{(|x|-R)^2}{8t_0}} e^{-\frac{2R|x|}{8t_0}},
    \end{align*}
    for all $\tau \in ]0, t[$, $y \in \partial B(0, R)$ and $R > 0$ big enough so that $x \in B(0, R-1)$. Accordingly, the previous
    inequality imply that the the second term in the right hand side of \eqref{repf1} tends to $0$ as $R$ goes to $+\infty$.
    Next, we consider the first term in the right hand side of \eqref{repf1}. We note that by the interior estimates of Lemma \ref{intest}, if we fix a $r \in ]0,1[$ such that $r^2 < t$ and $B(x, r) \subseteq \Omega^-$ (those conditions are always uniformly satisfied for $R$ big enough), then we have that
    \begin{align*}
        |\partial_{x_i} u(t,x)| 
        &\leq C_{0,1} r^{-n-3} \int_{t-r^2}^{t} \int_{B(x,r)} |u(\tau,y)| \,dy\,d\tau
        \leq C_{0,1} r^{-n-3}\int_{t-r^2}^{t} \int_{B(x,r)} C e^{c|y|^2} \,dy\,d\tau
        \\
        &\leq C_{0,1} r^{-n-3} \int_{t-r^2}^{t} \int_{B(x,r)} e^{c(|x|+r)^2} \,dy\,d\tau 
        \leq C_{0,1} r^{-n-3} C |B(0,1)|  e^{c(|x|+r)^2} r^{2+n}
        \\
        & \leq C_{0,1} C |B(0,1)| r^{-1} e^{c(|x|+r)^2}
        \leq C_{0,1} C |B(0,1)| r^{-1} e^{2c|x|^2}.
    \end{align*}

    Accordingly, the first derivative of $u$ enjoy the same growth bound at infinity as $u$. Thus, up to taking a smaller $t_0$, using \eqref{phi1eq} instead of \eqref{phi2eq} we can prove that also the first term in the right hand side of \eqref{repf1} tends to $0$ as $R$ goes to $+\infty$ for all $(t, x) \in ]0, t_0] \times \Omega^-$. In other words, we have proved the following representation formula for $u$:
    \begin{equation}\label{repf2}
        u(t,x) = \int_{0}^{t} \int_{\Omega^-} S_n (t-\tau,x-y) \left( 2 \nabla u(\tau,y) \cdot \nabla \xi(y) + u(\tau,y) \Delta \xi(y) \right) \, dy \, d\tau,
    \end{equation}
    for all $t \in ]o,t_0]$ with $t_0$ small enough, and for all $x$ in a neighborhood of infinity. We note that the integrand of the formula \eqref{repf2} has compact support, say $V$. Then, for $t \in ]0, t_0]$ and for $x$ in a neighborhood of infinity we have
    \begin{equation*}
        |u(t,x)| \leq \|2 \nabla u \cdot \nabla \xi + u \Delta \xi \|_{C^0([0,t_0]\times V)} \int_{0}^{t} \int_{V} S_n (t-\tau,x-y) \, dy \, d\tau,
    \end{equation*}
    and accordingly it is easily seen that there exist two constant $A, B > 0$ such that
    \begin{equation}\label{expd1}
        |u(t,x)| \leq A e^{-B|x|^2} \qquad \forall (t,x) \in ]0,t_0] \times \Omega^-,
    \end{equation}
    that is $u$ has actually an exponential decay at infinity. Indeed 
    \begin{align*}
        \int_{0}^{t} \int_{V} S_n (t-\tau,x-y) \, dy \, d\tau 
        &\leq K_1 \int_{0}^{t} \int_{V} e^{-\frac{|x-y|^2}{8t_0}} \, dy \, d\tau 
        \leq K_1 \int_{0}^{t} \int_{V} e^{-\frac{(|x|-|y|)^2}{8t_0}} \, dy \, d\tau 
        \\
        &\leq K_1 e^{-\frac{|x|^2}{8t_0}} e^{-\frac{(\inf_{y \in V} |y|)^2}{8t_0}} e^{\frac{2 |x| (\sup_{y \in V} |y|)}{8t_0}} \int_{0}^{t} \int_{V} 1 \, dy \, d\tau
    \end{align*}
    for $t \in ]0, t_0]$ and for $x$ in a neighborhood of infinity (notice that $V$ is compact). Moreover, up to taking a smaller $t_0$, since also the first derivatives of $u$ solve the heat equation and since we didn't use in this first part the boundary behavior or the boundary regularity of $u$, we can repeat the same argument and prove that also
    \begin{equation}\label{expd2}
        |\nabla u(t,x)| \leq A e^{-B|x|^2} \qquad \forall (t,x) \in ]0,t_0] \times \Omega^-.
    \end{equation}
    We are now ready to use an energy argument. Let $R$ big enough. Let
    \begin{equation*}
        e_R(t) := \int_{B(0,R) \setminus\overline{\Omega}} (u(t,x))^2 \,dy \qquad \forall t \in ]0,t_0[.
    \end{equation*}
    We can demonstrate that $e_R$ belongs to $C^1(]0,t_0[)$. A detailed proof is provided in \cite[Lem. 5 and Prop. 2]{Lu18}, and it is based on classical differentiation theorems for integrals depending on a parameter, along with a specific approximation of the support of integration (see Verchota \cite[Thm. 1.12, p. 581]{Ve84}). Then, Divengence Theorem and the argument in \cite{Lu18} imply that
    \begin{align*}
        \frac{d}{dt} e_R(t) &= - 2 \int_{B(0,R)\setminus \overline{\Omega}} |\nabla u(t,y)|^2 \,dy + 2 \int_{\partial B(0,R)} u(t,y) \frac{\partial}{\partial\nu_{B(0,R)}(y)} u(t,y) \,d\sigma_y
        \\
        & \quad - 2 \int_{\partial \Omega} u(t,y) \frac{\partial}{\partial\nu_{\Omega}(y)} u(t,y) \,d\sigma_y
        \\
        & = - 2 \int_{B(0,R)\setminus \overline{\Omega}} |\nabla u(t,y)|^2 \,dy + 2 \int_{\partial B(0,R)} u(t,y) \frac{\partial}{\partial\nu_{B(0,R)}(y)} u(t,y) \,d\sigma_y,
    \end{align*}
    where we have used that the normal derivative of $u$ is zero on $\partial\Omega$. Letting $R$ tend to infinity and using the exponential decay at infinity of \eqref{expd1} and \eqref{expd2}, we see that
    \begin{equation*}
        e(t) := \lim_{R \to +\infty} e_R(t) = \int_{\Omega^-} (u(t,y))^2 \,dy \qquad\forall t \in ]0,t_0[,       
    \end{equation*}
    and that
    \begin{equation*}
        \tilde{e}(t) := \lim_{R \to +\infty} \frac{d}{dt} e_R(t) =  -2 \int_{\Omega^-} |\nabla u(t,y)|^2 \,dy \qquad\forall t \in ]0,t_0[.      
    \end{equation*}
    Thus, $e$ is of class $C^1(]0, t_0[)$ and $\frac{d}{dt} e = \tilde{e}$ in $]0, t_0[$. Since $ \frac{d}{dt} e \leq  0$ in $]0, t_0[$, $e(0) = 0$ and $e \geq 0$ in $]0, t_0[$, we conclude that $e = 0$ in $[0, t_0]$ and accordingly $u = 0$ in $[0, t_0] \times \Omega^-$. Finally, as mentioned at the beginning of the proof, the general statement follows by a standard partition argument of the interval $[0,T]$ and repeating the above proof in each subintervals.
\end{proof}

\subsection{Uniqueness results for interior Dirichlet and Neumann problems}

We recall known fact about uniqueness for the interior Dirichlet and Neumann problems for the heat equation. The proof follows by a standard energy argument: a detailed proof is provided in \cite[Lem. 5 and Prop. 2]{Lu18} (see also Verchota \cite[Thm. 1.12, p. 581]{Ve84}).

\begin{theorem}\label{app thm uni int dir-neu}
    Let $\alpha \in ]0,1[$ and $T>0$. Let $\Omega$ be a bounded open subset of $\mathbb{R}^n$ of class $C^{1,\alpha}$. Let $u \in C_0^{\frac{\alpha}{2};  \alpha}([0,T] \times \overline{\Omega})$ (resp. $u \in C_0^{\frac{1+\alpha}{2};  1+\alpha}([0,T] \times \overline{\Omega})$) be a solution of the following interior problem
    \begin{equation*}
	\begin{cases}
	\partial_t u - \Delta u = 0 & \quad \text{in } [0,T] \times \overline{\Omega}, 
    \\
    u(0,\cdot)= 0 & \quad \mbox{in } \overline{\Omega},
	\end{cases}
	\end{equation*}
    which also satisfies the following condition:
    \begin{equation*}
    u = 0  \quad \mbox{on } [0,T] \times \partial \Omega \quad \left(resp.\quad \frac{\partial}{\partial \nu_\Omega} u= 0
    \quad \mbox{on } [0,T] \times \partial \Omega \right).
    \end{equation*}
    Then $u=0$ in $[0,T] \times \overline{\Omega}$.
\end{theorem}

\subsection*{Acknowledgment}

The author is member of the ``Gruppo Nazionale per l'Analisi Matematica, la Probabilit\`a e le loro Applicazioni'' (GNAMPA) of the ``Istituto Nazionale di Alta Matematica'' (INdAM). The author is partially support by INdAM - GNAMPA Project ``Problemi ellittici e sub-ellittici: non linearit\`a, singolarit\`a e crescita critica'', CUP E53C23001670001. The author acknowledges the support of the project funded by the EuropeanUnion - NextGenerationEU under the National Recovery and Resilience Plan (NRRP), Mission 4 Component 2 Investment 1.1 - Call PRIN 2022 No. 104 of February 2, 2022 of Italian Ministry of University and Research; Project 2022SENJZ3 (subject area: PE - Physical Sciences and Engineering) ``Perturbation problems and asymptotics for elliptic differential equations: variational and potential theoretic methods''. The author wishes to thank M.\,Dalla Riva, P.\,Luzzini and P.\,Musolino for fruitful discussions.

\end{document}